\documentclass[AMA,Times1COL]{WileyNJDv5}

\usepackage{bm}
\usepackage{graphicx}
\usepackage{subcaption}
\usepackage{lscape}
\newcommand{\hyperface}{\mathcal{H}}
\newcommand{\flux}{\bm{\Phi}}

\articletype{Research Article}%

\received{Date Month Year}
\revised{Date Month Year}
\accepted{Date Month Year}
\journal{International Journal for Numerical Methods in Fluids
}
\volume{00}
\copyyear{2025}
\startpage{1}

\begin{document}

\title{Edge-based discretizations on triangulations in $\mathbb{R}^d$,  with special attention to four-dimensional space.}

\author[1]{Nicholas Tufillaro}

\author[1]{David M. Williams}

\author[2]{Hiroaki Nishikawa}

\authormark{TUFILLARO \textsc{et al.}}
\titlemark{Edge-based discretizations on triangulations in $\mathbb{R}^d$, with special attention to four-dimensional space}

\address[1]{\orgdiv{Mechanical Engineering}, \orgname{Pennsylvania State University}, \orgaddress{\state{Pennsylvania}, \country{United States}}}

\address[2]{\orgname{National Institute of Aerospace}, \orgaddress{\state{Virginia}, \country{United States}}}

% \address[3]{\orgdiv{Department Name}, \orgname{Institution Name}, \orgaddress{\state{State Name}, \country{Country Name}}}

\corres{Corresponding author David M. Williams, \email{david.m.williams@psu.edu}}

\presentaddress{136 Research East, Pennsylvania State University, University Park, Pennsylvania, 16802}

%\fundingInfo{Text}
%\JELinfo{ejlje}

\abstract[Abstract]{Many time-dependent problems in the field of computational fluid dynamics can be solved using \emph{space-time} methods. However, such methods can encounter issues with computational cost and robustness. In order to address these issues, efficient, node-centered edge-based schemes are currently being developed. In these schemes, a median-dual tessellation of the space-time domain is constructed based on an initial triangulation. These methods are \emph{node-centered} or \emph{node-based}, as the primary components of the discretization are median-dual regions (polytopes) which surround the mesh nodes. These methods are extremely robust, as the median-dual regions will often maintain a positive volume and other good geometric properties, even in cases when some of the associated simplices have negative volumes, or other poor geometric properties. Unfortunately, it is not straightforward to construct median-dual regions or deduce their properties on triangulations for $d \geq 3$. In this work, we  provide the first rigorous definition of median-dual regions on triangulations in any number of dimensions. In addition, we introduce a new method for computing the hypervolume of a median-dual region in $\mathbb{R}^d$. Furthermore, we provide a new approach for computing the directed-hyperarea vectors for faces of a median-dual region in $\mathbb{R}^{4}$. These geometric properties are key for developing node-centered edge-based schemes in any number of dimensions.  We conclude our work by analyzing the computational complexity of the edge-based schemes, and performing numerical experiments in two, three, and four dimensions. We successfully demonstrate their effectiveness by obtaining accurate solutions to several canonical problems.}


\jnlcitation{\cname{%
\author{Tufillaro N},
\author{Williams D M}, and
\author{Nishikawa H}}.
\ctitle{Edge-based discretizations on triangulations in $\mathbb{R}^d$, with special attention to four-dimensional space} \cjournal{\it Int. J. Num. Meth. Flu.} \cvol{2025;00(00):1--27}.}

\maketitle

\renewcommand\thefootnote{}
\footnotetext{\textbf{Abbreviations:} CFD, computational fluid dynamics; MDR, median-dual region; RK, Runge-Kutta; CFK, Coxeter-Freudenthal-Kuhn.}

\renewcommand\thefootnote{\fnsymbol{footnote}}
\setcounter{footnote}{1}

\section{Introduction}

Node-centered methods and cell-centered methods have become popular tools for numerically solving partial differential equations. In a node-centered method, the unknowns (degrees of freedom) are the vertices of a triangulation, and in a cell-centered method, the unknowns are the cell averages. Here, a triangulation is defined as a mesh composed of $d$-simplicial elements which covers a domain $\Omega \subset \mathbb{R}^d$. We will refer to the triangulation as the \emph{primitive} or \emph{primary} tessellation of the domain. Evidently, the primitive tessellation is naturally associated with cell-centered methods. In contrast, the \emph{dual} tessellation is a mesh composed of $d$-polytopes which covers the domain, where the polytopes contain the triangulation vertices (nodes) as their `centers'. Evidently, the dual tessellation is associated with node-centered methods. The dual regions---often called dual control volumes---provide a convenient way to enforce conservation of mass, momentum, or energy in the immediate vicinity of the nodes. The idea of conservation is important in the field of computational fluid dynamics (CFD), as conservation errors can result in incorrect predictions of shockwave speeds and thermodynamic properties~\cite{part1994conservative,abgrall2001computations,johnsen2012preventing,lv2014discontinuous,peyvan2022oscillation,gaburro2024discontinuous}. Of course, node-centered numerical methods are not the only methods which enforce conservation; evidently, cell-centered methods accomplish the same task. However, node-centered methods have become increasingly popular due to their exceptional efficiency. For example,  \emph{node-centered} \emph{edge-based} methods have become the foundation of many important CFD codes~\cite{barth_AIAA1991,fun3d_manual:NASA20240006306,nakashima_watanabe_nishikawa:Japan2014,dlr-tau-digital-x,mavriplis_long:AIAA2010,Luo_Baum_Lohner:AIAA2004-1103,KozubskayaAbalakinDervieux:AIAA2009,Haselbacher_Blazek_AIAAJ2000,sierra-primo:AIAA2002,Eliasson_EDGE:2001,fezoui_stoufflet:JCP1989,GaoHabashiFossatiIsolaBaruzzi_AIAA2017-0085}. These methods are remarkably inexpensive, as they compute the residual using a compact stencil that only involves the solution and gradients at the current node and its edge-neighbors. Conservation is enforced via careful calculations of numerical fluxes at the dual faces which intersect the edges. In addition, fluxes through dual faces over the boundaries are added with special accuracy-preserving formulas \cite{nishikawa:AIAA2010,nishikawa_boundary_quadrature:JCP2015}; furthermore, boundary conditions are enforced through an upwind numerical flux with the external state defined by a physical boundary condition~\cite{liu_nishikawa_aiaa2016-3969}. Generally speaking, node-centered, edge-based methods are quite fast, as they only require a single loop over edges in order to compute the residuals, and a single loop over boundary elements to close the residuals at boundary nodes and enforce boundary conditions. Most importantly, node-centered, edge-based methods are capable of achieving second-order or third-order accuracy on unstructured triangulations in 2D~\cite{Katz_Sankaran_JCP:2011,katz_sankaran:JSC_DOI,diskin_thomas:AIAA2012-0609} and 3D~\cite{liu_nishikawa_aiaa2016-3969,liu_nishikawa_aiaa2017-0081}. We note that third-order accuracy can be achieved with a single numerical flux per edge, without requiring quadratic (curved) triangulations~\cite{liu_nishikawa_aiaa2016-3969,nishikawa_boundary_quadrature:JCP2015}, or second-derivative information~\cite{nishikawa_liu_source_quadrature:jcp2017}.
%, or an expanded stencil~\cite{Nishikawa_aiaa_scitech2024}.
%although the gradient stencil may need to be expanded, for the sake of robustness, by one edge-neighbor in this case (or one can employ an implicit edge-based gradient method \cite{Nishikawa_aiaa_scitech2024} to avoid expanding the stencil). 

In principle, a node-centered, edge-based scheme can be employed on \emph{any} dual tessellation, subject to the following geometric constraints: 
\begin{enumerate}[(a)]
    \item The dual cell for a given node must contain the node.
    \item  The dual cell must be contained within the union of all $d$-simplices which share the node.
    \item  The dual cell must be simply connected.
    \item The aggregate hypervolumes of the dual tessellation and the original triangulation must be identical. 
\end{enumerate}
The first two constraints help ensure proximity of the dual cell and the associated node, 
%(the dual cell is not separated from the node by some, uncontrolled distance) 
and sizing (scale) of the dual cell relative to the initial triangulation.
%(the size of the dual region is directly controlled by the original mesh spacing of the triangulation). 
 These constraints can be viewed as \emph{locality} constraints, which require the immediate region of influence of a node to coincide with its immediate geometric neighborhood. The locality constraints are generally desired, if not strictly required, because violation of these constraints may imply the existence of inverted elements. Failing to satisfy these constraints is undesirable, but not necessarily fatal, as there is evidence that the edge-based methods work just fine even if some elements have zero or negative volumes~\cite{nishikawa_aiaa2017-4295}.

 The last two constraints help enforce conservation. While the locality constraints were somewhat flexible, the conservation constraints are non-negotiable, due to the importance of maintaining the correct quantities of mass, momentum, and energy.
%Evidently, if the dual cells are meant to enforce conservation in the immediate vicinity of a given node (and globally over the entire domain), they should satisfy the above constraints.

Of course, most well-known dual tessellations satisfy the constraints above. However, \emph{median-dual} tessellations are the most popular. First of all, this is because the existing node-centered edge-based method is second- or third-order accurate only with median-dual tessellations on simplex-element grids, and first-order accurate otherwise (see, e.g.~\cite{Nishikawa_aiaa2020-1786}). More generally, there are important reasons based on geometric considerations. These considerations are discussed in more detail below.

%In order to explain the reasoning behind this fact, we will provide a brief review of dual tessellations in what follows. 

\subsection{Geometric Background}

A median-dual tessellation of a triangulation in $\mathbb{R}^{d}$ is a set of $d$-polytopal cells which covers the domain of triangulation, where the vertices of the dual regions are the centroids of 2-simplices, 3-simplices, ... , and $d$-simplices which belong to the triangulation, and share a given node. 
The resulting regions are simply connected, possess straight edges, always contain the central node of interest, and are (often) non-convex. By construction, the regions possess the same aggregate hypervolume as the original triangulation. It is important to note that median-dual regions are not (usually) identical to \emph{centroid-dual} regions, see~\cite{barth1992aspects}. For example, in 2D, centroid-dual regions are formed by connecting the centroids of adjacent 2-simplices (triangles) which share a given node. A line segment connecting the centroids of adjacent 2-simplices rarely crosses the shared edge at the midpoint; therefore, the face intersection points for the median-dual and centroid-dual regions are often different. %Figure~\ref{fig:median_centroid_compare} highlights the differences between a median-dual region and a centroid-dual region in 2D. 
Unfortunately, centroid-dual regions violate one of our geometric constraints: namely, they can fail to contain the central node of interest. Figure~\ref{fig:centroidproblem} shows an example of this issue. 
\begin{figure}[h!]
    \centering
    \includegraphics[width = 0.5\textwidth]{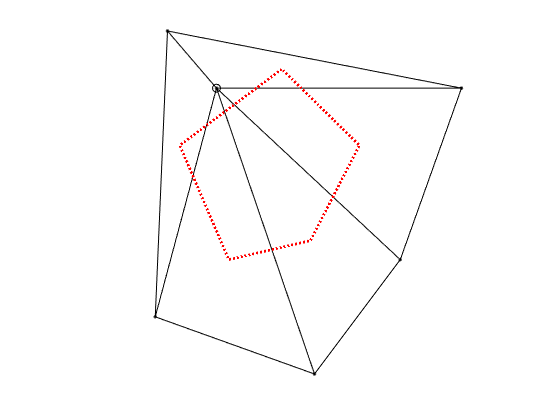}
    \caption{An example of the centroid-dual region not containing the associated node in 2D. The associated node is denoted by an open circle, and the region is denoted with a dotted-red line.} %Adjacent nodes are denoted with small solid circles to avoid confusion.}
    \label{fig:centroidproblem}
\end{figure}

With the above discussion in mind, it makes sense to consider developing dual tessellations based on non-centroidal quantities, such as incenters, hypercircumcenters, or orthocenters. Each of these geometric quantities generalizes to higher dimensions, $d \geq 3$. 
%In particular, the incenter of a $d$-simplex is the center of the largest hypersphere that can be contained in the simplex. In addition, the  hypercircumcenter of a $d$-simplex is the center of the unique hypersphere which passes through all of the $(d+1)$ vertices of the simplex. Lastly, the orthocenter of a $d$-simplex is the intersection of the altitudes of the simplex. 
Unfortunately, some of these alternative options violate our geometric constraints. For example, the hypercircumcenters and orthocenters of $d$-simplices are not guaranteed to reside within the interiors of the simplices. Even in 2D, the circumcenter and orthocenter of an obtuse 2-simplex lie outside of the simplex~\cite{coxeter1967geometry}. Fortunately, by construction, the incenter lies inside of the $d$-simplex. 
However, the computation of the incenter for a $d$-simplex often requires the solution of a linear system of equations, (see for example~\cite{klein2020insphere}). It is unclear that one gains any advantage from this additional complexity. %Figure~\ref{fig:median_incenter_compare} highlights the differences between a median-dual region and an incenter-dual region in 2D. 

Our discussion of dual tessellations would be incomplete without mentioning \emph{Voronoi} and \emph{Laguerre} tessellations. The Voronoi tessellation (or Voronoi diagram)~\cite{voronoi1908nouvellesA,voronoi1908nouvellesB}, is a space-efficient subdivision of a domain, in which each dual cell contains all the points in the domain which are closer to the central node of the cell than they are to any other node in the triangulation. The Voronoi dual cells are convex, simply connected, and possess straight edges.  It is well-known that Voronoi tessellations are the duals of Delaunay triangulations---assuming all the sites of the Voronoi cells are in general position~\cite{boissonnat2018geometric}. 
%The vertices of  Voronoi cells are the hypercircumcenters of Delaunay $d$-simplices. 
In a similar fashion, a Laguerre tessellation (or Power diagram)~\cite{blaschke1929vorlesungen,aurenhammer1987power} is the dual of a weighted-Delaunay triangulation~\cite{boissonnat2018geometric}. The nodes of a weighted-Delaunay triangulation can be treated like hyperspheres with different radii that depend on the weights; in turn, the Laguerre dual cell for a given hypersphere contains all the points which are closer to the central hypersphere than they are to any other hypersphere in the weighted triangulation. 
%In this case, \emph{closeness} is measured via the power distance. 
Unfortunately, most algorithms which construct Voronoi or Laguerre tessellations must start with a Delaunay or weighted-Delaunay triangulation~\cite{bowyer1981computing,watson1981computing}, (with the exception of Fortune's algorithm~\cite{fortune1986sweepline}). This inhibits the use of these tessellations for non-Delaunay triangulations. 
%However, one may still construct both median-dual and Voronoi tessellations on a given Delaunay triangulation. %Figure~\ref{fig:median_voronoi_compare} highlights the differences between a median-dual region and a Voronoi region on a 2D Delaunay triangulation.

For the sake of completeness, Figures~\ref{fig:median_compare} and \ref{fig:dual_summary} highlight the differences between median-dual, centroid-dual, incenter-dual, and Voronoi regions on a 2D Delaunay triangulation.
\begin{figure}[h!]
    \centering
    \begin{subfigure}[t]{0.45\textwidth}
        \centering
    \includegraphics[height=2.5in,trim={2cm 0 2cm 0},clip]{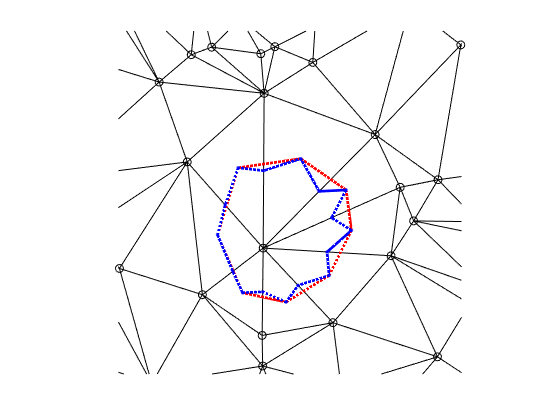}
    \subcaption{Median-dual and centroid-dual.}
    \label{fig:median_centroid_compare}
     \end{subfigure}
     \begin{subfigure}[t]{0.45\textwidth}
        \centering
    \includegraphics[height=2.5in,trim={2cm 0 2cm 0},clip]{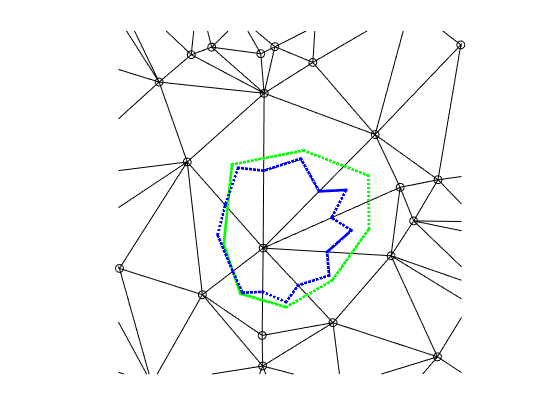}
    \subcaption{Median-dual and incenter-dual.}
    \label{fig:median_incenter_compare}
     \end{subfigure} 
     \begin{subfigure}[t]{0.45\textwidth}
        \centering
    \includegraphics[height=2.5in,trim={2cm 0 2cm 0},clip]{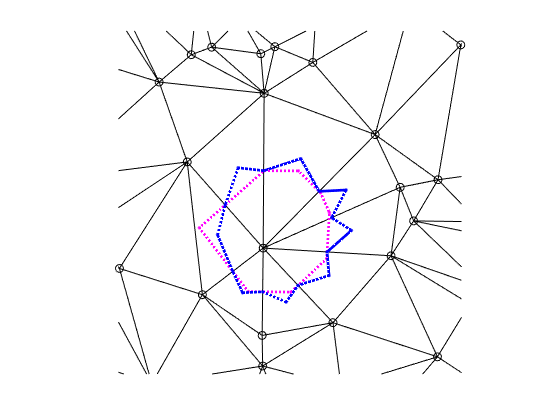}
    \subcaption{Median-dual and Voronoi.}
    \label{fig:median_voronoi_compare}
    \end{subfigure}
    \caption{An illustration comparing the median-dual region (dotted-blue line) to the centroid-dual region (dotted-red line, subfigure a), incenter-dual region (dotted-green line, subfigure b), and Voronoi region (dotted-magenta line, subfigure c) for a generic node of a Delaunay triangulation in~2D.}
    \label{fig:median_compare}
\end{figure}
\begin{figure}[h!]
    \centering
    \begin{subfigure}[t]{0.45\textwidth}
        \centering
    \includegraphics[height=2.5in,trim={2cm 0 2cm 0},clip]{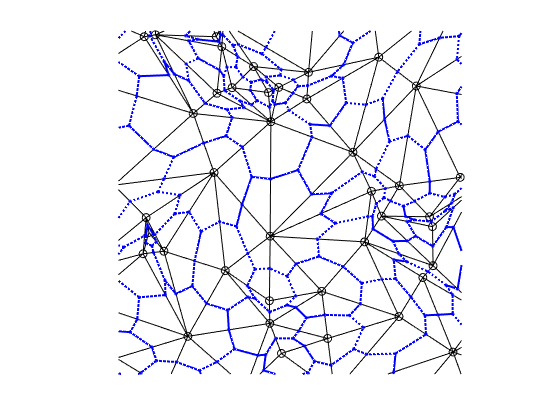}
    \subcaption{Median-dual regions.}
     \end{subfigure}
     \begin{subfigure}[t]{0.45\textwidth}
        \centering
    \includegraphics[height=2.5in,trim={2cm 0 2cm 0},clip]{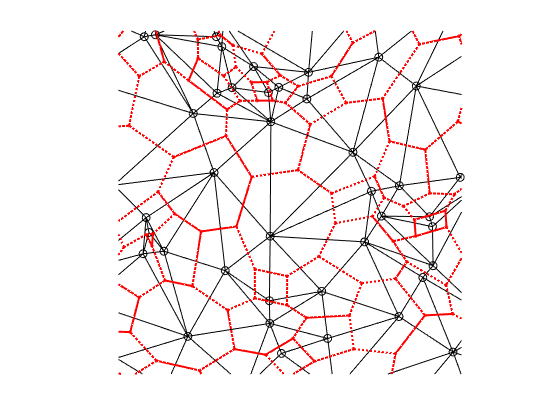}
    \subcaption{Centroid-dual regions.}
     \end{subfigure} 
     \begin{subfigure}[t]{0.45\textwidth}
        \centering
    \includegraphics[height=2.5in,trim={2cm 0 2cm 0},clip]{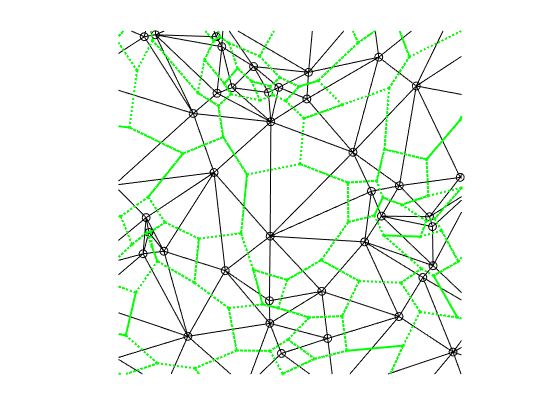}
    \subcaption{Incenter-dual regions.}
    \end{subfigure}
    \begin{subfigure}[t]{0.45\textwidth}
        \centering
    \includegraphics[height=2.5in,trim={2cm 0 2cm 0},clip]{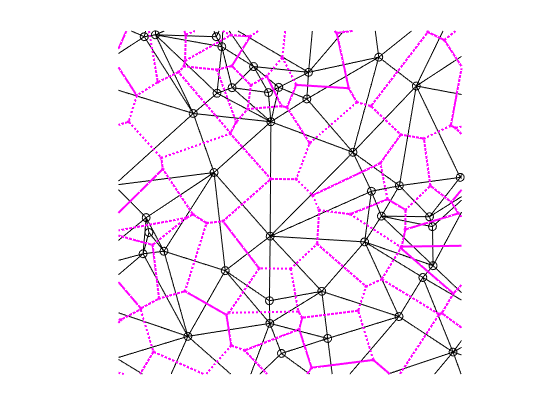}
    \subcaption{Voronoi regions.}
    \end{subfigure}
    \caption{An illustration showing the median-dual regions (dotted-blue lines), centroid-dual regions (dotted-red lines), incenter-dual regions (dotted-green lines), and Voronoi regions (dotted-magenta lines) for multiple nodes, on a Delaunay triangulation in 2D.}
    \label{fig:dual_summary}
\end{figure}

In summary, the key advantages of a median-dual tessellation are flexibility and simplicity. In particular, a median-dual tessellation can be formed based on \emph{any} valid triangulation of a domain in $\mathbb{R}^{d}$, without requiring the solutions of linear systems. This does not generally hold true for centroid-dual, incenter-dual, Voronoi, or Laguerre tessellations, or variants thereof. We note that flexibility is especially important, as there are many alternative strategies for building triangulations in $\mathbb{R}^{d}$, including the advancing-front approach~\cite{lohner1988generation,george1994advancing,moller1995advancing,chan1997automatic,lo2013dynamic}, and the hybrid Delaunay/advancing-front approach~\cite{mavriplis1995advancing,frey19983d}. In addition, even if the original triangulation of interest satisfies the Delaunay criterion, it is often necessary to adaptively refine the triangulation in order to resolve important features of the solution~\cite{neumuller2011refinement,caplan2019four,caplan2020four,belda2023conformal}. Generally speaking, the Delaunay criterion is not satisfied after this adaptive process is complete. As a result, we require dual tessellation strategies that are capable of operating on generic triangulations which emerge from such processes. 
%In light of this requirement, the choice of a median-dual tessellation is a natural one.  

\subsection{Recent Developments}

In previous years, CFD practitioners have often implemented node-centered, edge-based methods by \emph{explicitly} constructing median-dual regions (e.g.~median-dual vertices, edges, faces, and volumes in 2D or 3D). However, Nishikawa~\cite {Nishikawa:ijnmf2024_inreview} recently discovered that this is unnecessary.
%It turns out that node-centered, edge-based methods do \emph{not} require a complete construction of the median-dual regions, and their vertices.
Rather, it is possible to develop a node-centered scheme which only requires certain geometric properties of the median-dual regions: namely, the hypervolumes of the dual regions, as well as the directed-hyperarea vectors associated with the edges~\cite{Nishikawa:ijnmf2024_inreview}. In 3D, the directed-hyperarea vectors are area-weighted vectors which point from the current node towards neighboring nodes. Furthermore, the area weighting is associated with the area of the dual face through which each edge crosses. With these definitions established, it is desirable to identify algebraic formulas for computing the dual hypervolume and directed-hyperarea vectors, which are entirely based on geometric properties of the $d$-simplices which share a given node. This avoids the complex task of explicitly constructing the median-dual regions, and then extracting the necessary geometric quantities. In 2D and 3D, Nishikawa has recently obtained the desired formulas~\cite{Nishikawa:ijnmf2024_inreview}. However, it remains for us to extend these formulas to any number of dimensions $d \geq 2$, and rigorously prove their correctness. 

%This is the main objective of the present paper. 

\subsection{Overview of the Paper}

In this article, our objective is to develop and validate new geometric formulas for node-centered, edge-based schemes. The main contributions of this article can be summarized as follows:
\begin{itemize}
    \item We construct a new, rigorous definition for median-dual regions in $\mathbb{R}^{d}$. Previous definitions were only valid in 2D and 3D.
    \item We develop explicit formulas for the hypervolumes and directed-hyperarea vectors of median-dual regions in $\mathbb{R}^d$. 
    \item For the first time, we prove that our formula for the dual hypervolume holds in any number of dimensions.
    \item For the first time, we prove that our directed-hyperarea vector formula holds in~4D.
    \item  We present the  results of new numerical experiments on the edge-based scheme in 4D. 
\end{itemize}
 
The format of the paper is as follows. In Section~\ref{prelim_section}, we introduce preliminary geometric concepts and define the median-dual region in $\mathbb{R}^d$. In Section~\ref{theoretical_properties_d_section}, we  prove an explicit hypervolume formula in $\mathbb{R}^d$. Next, in Section~\ref{theoretical_properties_four_section}, we  prove an explicit directed-hyperarea formula in $\mathbb{R}^4$, and discuss implementation and verification procedures. In Section~\ref{complexity_section}, we discuss how the edge-based schemes can be applied to linear advection problems, and thereafter perform a complexity analysis. In Section~\ref{results_section}, we present the results of  numerical experiments on linear advection problems. Finally, in Section~\ref{conclusion_section}, we present some concluding remarks and propose future research directions.

\section{Preliminaries} \label{prelim_section}

Consider a set of distinct points in general position
\begin{align}
    \mathcal{P} = \left\{\bm{p}_{1}, \bm{p}_{2}, \ldots, \bm{p}_{\ell}, \ldots, \bm{p}_{L} \right\},
\end{align}
where $L$ is the total number of points, and $1\leq \ell \leq L$. 
% Since the points are distinct, we are guaranteed that the following inequality holds
% %
% \begin{align}
%     \forall \bm{p}_{\ell}, \bm{p}_{m} \in \mathcal{P}: \qquad \left|\bm{p}_{\ell} - \bm{p}_{m} \right| \geq \eta,
% \end{align}
% %
% for some $\eta > 0$, and $1 \leq m \leq L$. 
Now, suppose that we are interested in forming a triangulation of $\mathcal{P}$. Towards this end, let us define a domain $\Omega = \mathrm{conv}(\mathcal{P})$, where $\mathrm{conv}(\cdot)$ denotes the convex hull. A valid triangulation $\mathcal{T}$ of the domain $\Omega$ is one which covers $\Omega$ with non-overlapping $d$-simplicial elements, such that the sum of the hypervolumes of the elements is identical to the hypervolume of $\Omega$ itself. In more technical terms, $\mathcal{T}$ is a \emph{pseudo manifold}, i.e.~a pure simplicial $d$-complex that is $d$-connected,  and for which each $(d-1)$-face has one or two $d$-simplicial neighbors. Such a triangulation can be formed in a straightforward fashion by using one of the techniques previously mentioned in the Introduction.

For the sake of completeness, let us denote $T^{(d)}$ as the generic $d$-simplex that belongs to~$\mathcal{T}$. This simplex is the convex hull of $d+1$ points
\begin{align}
   \nonumber T= T^{(d)} &= \mathrm{conv} \left( \bm{p}_{T,1}, \bm{p}_{T,2}, \ldots, \bm{p}_{T,n} ,\ldots, \bm{p}_{T,d+1} \right), \\[1.0ex]
   &= \mathrm{conv} \left(\bm{p}_{i_1}, \bm{p}_{i_2}, \ldots, \bm{p}_{i_n}, \ldots \bm{p}_{i_{d+1}}\right), \label{simplex_def}
\end{align}
where $1 \leq n \leq (d+1)$. We omit the superscript $(d)$ when our meaning is clear. In Eq.~\eqref{simplex_def}, the quantity $\bm{p}_{T,n}$ is a point labeled with the \emph{local numbering} of simplex $T^{(d)}$, and the quantity $\bm{p}_{i_n}$ is the same point labeled with the \emph{global numbering} of set $\mathcal{P}$.

Once each element $T^{(d)}$ has been constructed, we can define the appropriate \emph{median} quantities, as follows. 

\subsection{Vertex-Based Quantities}
We start by introducing the set of all edges in the mesh that start at point $\bm{p}_{j}$ and end at adjacent points $\bm{p}_{1}$, $\bm{p}_{2}, \ldots, \bm{p}_{k}, \ldots, \bm{p}_{M_1}$
\begin{align}
    \nonumber \mathbb{T}^{(1)}_{j} &\equiv \left\{\bm{p}_{1} - \bm{p}_{j}, \bm{p}_{2} - \bm{p}_{j}, \ldots, \bm{p}_{k} - \bm{p}_{j}, \ldots, \bm{p}_{M_1} - \bm{p}_{j} \right\}  \\[1.0ex]
    &=\left\{T_{j,1}^{(1)}, T_{j,2}^{(1)}, \ldots, T_{j,k}^{(1)}, \ldots T_{j,M_1}^{(1)}  \right\},
\end{align}
where $1\leq k\leq M_1$ and $M_1 = \mathrm{card}\left(\mathbb{T}^{(1)}_{j} \right)$ is the total number of points which are adjacent (edge-wise) to $\bm{p}_{j}$. In addition, we can denote the centroid of each edge as follows
\begin{align}
    \nonumber \mathcal{C}^{(1)}_{j} &\equiv \frac{1}{2} \left\{\bm{p}_{1} + \bm{p}_{j}, \bm{p}_{2} + \bm{p}_{j}, \ldots, \bm{p}_{k} + \bm{p}_{j}, \ldots, \bm{p}_{M_1} + \bm{p}_{j} \right\} \\[1.0ex] 
    &=\left\{\bm{c}_{j,1}^{(1)}, \bm{c}_{j,2}^{(1)}, \ldots, \bm{c}_{j,k}^{(1)}, \ldots, \bm{c}_{j,M_1}^{(1)}  \right\}.
\end{align}
Next, we introduce the set of triangles (2-simplices) that contain the point $\bm{p}_{j}$
\begin{align}
    \mathbb{T}^{(2)}_{j} &\equiv \left\{T_{j,1}^{(2)}, T_{j,2}^{(2)}, \ldots, T_{j,k}^{(2)}, \ldots, T_{j,M_2}^{(2)} \right\},
\end{align}
along with their centroids
\begin{align}
    \mathcal{C}^{(2)}_{j} & \equiv \left\{\bm{c}_{j,1}^{(2)}, \bm{c}_{j,2}^{(2)}, \ldots, \bm{c}_{j,k}^{(2)}, \ldots, \bm{c}_{j,M_2}^{(2)} \right\},
\end{align}
where $1 \leq k \leq M_2$ and $M_2 = \mathrm{card}\left(\mathbb{T}^{(2)}_{j} \right)$ . More generally, we introduce the set of $q$-simplices which share the point $\bm{p}_{j}$
\begin{align}
    \mathbb{T}^{(q)}_{j} &\equiv  \left\{T_{j,1}^{(q)}, T_{j,2}^{(q)}, \ldots, T_{j,k}^{(q)}, \ldots, T_{j,M_{q}}^{(q)} \right\}, 
\end{align}
along with their centroids
\begin{align}
    \mathcal{C}^{(q)}_{j} &\equiv  \left\{\bm{c}_{j,1}^{(q)}, \bm{c}_{j,2}^{(q)}, \ldots, \bm{c}_{j,k}^{(q)}, \ldots, \bm{c}_{j,M_{q}}^{(q)} \right\}, 
\end{align}
where $1 \leq k \leq M_q$, $M_q = \mathrm{card}\left(\mathbb{T}^{(q)}_{j} \right)$, and $1 \leq q \leq d$.  Note that we can compute the centroid of a $q$-simplex as the arithmetic average of its vertex coordinates, e.g.
\begin{align}
    \bm{c}_{j,k}^{(q)} = \frac{1}{q+1} \sum_{\bm{p} \subset T_{j,k}^{(q)}} \bm{p}.
\end{align}

We have introduced all of the relevant vertex-based quantities (see above). These quantities have been partitioned into convenient sets, denoted by (for example) $\mathbb{T}_{j}^{(q)}$ and $\mathcal{C}_{j}^{(q)}$. In a natural fashion, we can also construct larger sets, which are based on the \emph{power set} of the smaller sets. In particular, we can define the supersets of all simplices and centroids that share the point $\bm{p}_{j}$ as follows
\begin{align}
    \mathbb{T}_{j} &\equiv \left\{  \mathbb{T}^{(1)}_{j},  \mathbb{T}^{(2)}_{j}, \ldots,  \mathbb{T}^{(q)}_{j}, \ldots,  \mathbb{T}^{(d)}_{j} \right\}, \qquad
    \mathcal{C}_{j} \equiv \left\{  \mathcal{C}^{(1)}_{j},  \mathcal{C}^{(2)}_{j}, \ldots,  \mathcal{C}^{(q)}_{j}, \ldots,  \mathcal{C}^{(d)}_{j} \right\}.
\end{align}

\subsection{Median-Dual Definitions}

We are now ready to formulate the definition of the median-dual region.

\begin{definition}[Median-Dual Region]
    Consider a point $\bm{p}_j$ in a $d$-simplicial triangulation. We can form the median-dual region for $\bm{p}_j$ by taking the union of the convex hulls of the centroids (in the superset $\mathcal{C}_j$) which belong to each $d$-simplex $T_{j,k}^{(d)}$, where each simplex shares the vertex $\bm{p}_j$ and $1 \leq k \leq M_{d}$. Equivalently, the median-dual region (MDR) around $\bm{p}_j$ is given by the following
    \begin{align}
       \mathrm{MDR}(\bm{p}_j) &= \bigcup_{k=1}^{M_d} \mathrm{conv}\left( \bm{c} \in \left( \mathcal{C}_{j} \cap T_{j,k}^{(d)} \right) \right), \label{median_def_eq}
    \end{align}
    where $\bm{c}$ is a generic centroid that belongs to the intersection of $\mathcal{C}_j$ and $T_{j,k}^{(d)} \in \mathbb{T}^{(d)}_{j}$.  In general, the centroids $\bm{c}$ are not just volumetric centroids, but they are also the centroids of $(d-1)$-faces, $(d-2)$-faces, $\ldots$, and $1$-faces (edges) which share the vertex $\bm{p}_j$.
    %The median-dual  region formed in this fashion is a star-shaped region with respect to the point~$\bm{p}_j$. 
    \label{mdr_def_one}
\end{definition}

We can now introduce a very rudimentary formula for the hypervolume of each median-dual region. This result is required in order to develop a subsequent, more sophisticated hypervolume identity (see Theorem~\ref{hypervolume_theorem}).

\begin{lemma}
    [Median-Dual Hypervolume]
    The hypervolume of the median-dual region around $\bm{p}_{j}$ is given by the following formula
    \begin{align}
        V\left(\bm{p}_{j}\right) = \frac{1}{d+1} \sum_{k=1}^{M_d} \left| T_{j,k}^{(d)} \right|,
        \label{volume_def}
    \end{align}
    where $T_{j,k}^{(d)} \in \mathbb{T}_{j}^{(d)}$ is a $d$-simplex that shares the node $\bm{p}_j$,  and
    \begin{align}
        \left| T_{j,k}^{(d)} \right| = \mathrm{vol}\left(T_{j,k}^{(d)}  \right),
    \end{align}
    is the hypervolume of the simplex. A simple formula for the $d$-simplex hypervolume appears in~\cite{stein1966note}.
\end{lemma}

\begin{proof}
    By construction, the median dual region covers exactly $1/(d+1)$th of each $d$-simplex which contains the node $\bm{p}_{j}$. The summation of these partial hypervolume contributions from each $T_{j,k}^{(d)}$ yields the total hypervolume of the region.
\end{proof}

\subsection{Edge-Based Quantities}

We can define the set of all $q$-simplices that share the edge $\bm{p}_{k} - \bm{p}_{j}$
\begin{align*}
    \mathbb{T}_{jk}^{(q)} \equiv \left\{T_{jk,1}^{(q)}, T_{jk,2}^{(q)}, \ldots, T_{jk,v}^{(q)}, \ldots T_{jk,N_q}^{(q)}  \right\},
\end{align*}
where $1 \leq v \leq N_q$ and $N_q = \mathrm{card}\left( \mathbb{T}_{jk}^{(q)} \right)$.
The centroids of the $q$-simplices are given by
\begin{align*}
    \mathcal{C}_{jk}^{(q)} \equiv  \left\{\bm{c}_{jk,1}^{(q)}, \bm{c}_{jk,2}^{(q)}, \ldots, \bm{c}_{jk,v}^{(q)}, \ldots \bm{c}_{jk,N_q}^{(q)}  \right\}.
\end{align*}
The supersets of \emph{all} simplices and centroids ($q = 1, \ldots, d$)  that share the edge $\bm{p}_{k} - \bm{p}_{j}$ are given by
\begin{align}
    \nonumber \mathbb{T}_{jk} &= \left\{\mathbb{T}_{jk}^{(1)}, \mathbb{T}_{jk}^{(2)}, \ldots, \mathbb{T}_{jk}^{(q)}, \ldots, \mathbb{T}_{jk}^{(d)}  \right\}, \qquad \mathcal{C}_{jk} = \left\{\mathcal{C}_{jk}^{(1)}, \mathcal{C}_{jk}^{(2)}, \ldots, \mathcal{C}_{jk}^{(q)}, \ldots, \mathcal{C}_{jk}^{(d)}  \right\}.
\end{align}

\subsection{Directed-Hyperarea Vector Definition}

\begin{definition}[Directed-Hyperarea Vectors]

Consider a point $\bm{p}_j$ in a $d$-simplicial triangulation. The directed-hyperarea vectors are denoted by
\begin{align}
    \left\{\bm{n}_{j1}, \bm{n}_{j2}, \ldots, \bm{n}_{jk}, \ldots, \bm{n}_{jM_1}\right\}, 
\end{align}
 where the vector $\bm{n}_{jk}$ has a positive dot product with the edge vector from node $\bm{p}_j$ to node $\bm{p}_k$  
\begin{align}
\bm{n}_{jk} \cdot (\bm{p}_{k} - \bm{p}_{j}) > 0.
\end{align}
Furthermore
\begin{align}
    \bm{n}_{jk} \equiv \sum_{T \in\left\{\mathbb{T}^{(d)}_{jk} \right\}} \bm{n} \left( \hyperface^{T} \right),
    \label{cube_facet_normal}
\end{align}
where $\bm{n} \left( \hyperface^{T} \right)$ is the lumped-normal vector of the $(d-1)$-hypercuboid face that is constructed from centroids which are contained in each $T$; (see Figures~\ref{fig:directedhyperareavector} and \ref{fig:hypercuboidfacet} for an illustration of the quantities in the formula). More precisely, the lumped-normal vector of a hypercuboid is defined as the unique sum of the normal vectors of a simplicial tessellation of the hypercuboid.  For example in 3D, it is the sum of the normal vectors of two triangles which make up a quadrilateral. Here, the magnitude of each normal vector is the area of the corresponding triangle. 

Furthermore, each hypercuboid face is given by
\begin{align}
    \hyperface^{T} \equiv \mathrm{conv}\left( \bm{c} \in \left(\mathcal{C}_{jk} \cap T \right) \right),
\end{align}
where $\bm{c}$ is a generic centroid that belongs to the intersection of the superset $\mathcal{C}_{jk}$ and $T$, and $T$ is a generic $d$-simplex that belongs to
\begin{align*}
    \mathbb{T}_{jk}^{(d)} \equiv  \left\{T_{jk,1}^{(d)}, T_{jk,2}^{(d)}, \ldots, T_{jk,v}^{(d)}, \ldots T_{jk,N_d}^{(d)}  \right\}.
\end{align*}
Here, $\mathbb{T}_{jk}^{(d)}$ is the set of $d$-simplices that share the edge  $\bm{p}_{k} -\bm{p}_{j}$. 
% Equivalently, we can write
% %
% \begin{align}
%     \hyperface^{T} \equiv \mathrm{conv}\left( \bm{c} \in \left(\mathcal{C}_{jk} \cap T_{jk,v}^{(d)} \right) \right).
% \end{align}
\end{definition}

\begin{figure}[h!]
    \centering
    \includegraphics[width = 0.9\textwidth]{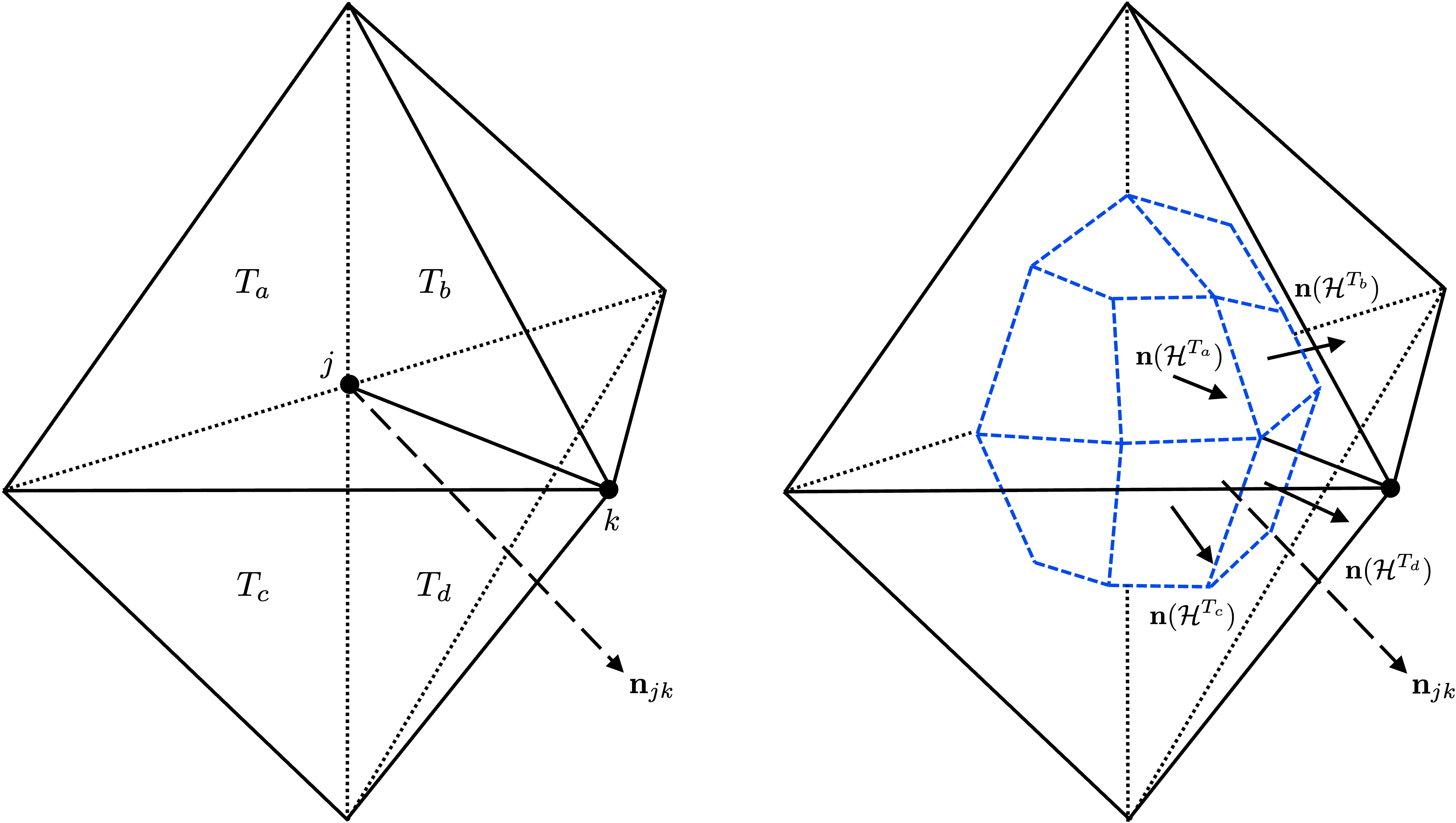}
    \caption{A 3D illustration of a directed-hyperarea vector $\bm{n}_{jk}$ for the edge $\bm{p}_{k}-\bm{p}_{j}$ (left). The edge is shared by four tetrahedra $T_a, T_b, T_c$ and $T_d$. The lumped-normal vectors of the dual hypercuboid faces are summed together in order to obtain $\bm{n}_{jk}$ (right). More precisely, $\bm{n}_{jk} = \bm{n}(\hyperface^{T_a}) + \bm{n}(\hyperface^{T_b}) + \bm{n}(\hyperface^{T_c}) + \bm{n}(\hyperface^{T_d})$. The (partial) median-dual region around node $\bm{p}_{j}$ is highlighted in blue.}
    \label{fig:directedhyperareavector}
\end{figure}

\begin{figure}[h!]
    \centering
    \includegraphics[width = 0.9\textwidth]{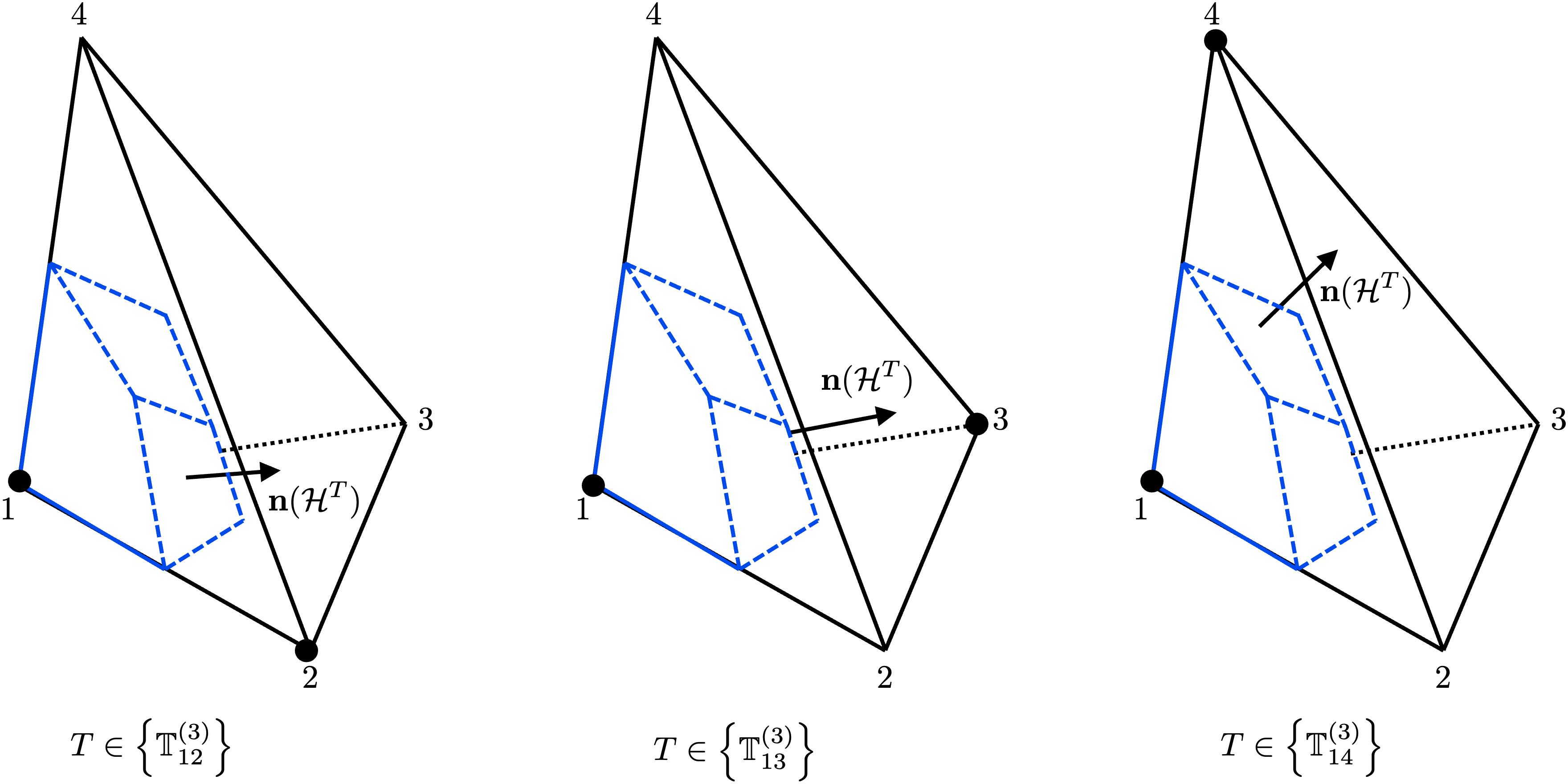}
    \caption{A 3D illustration of the lumped-normal vector of the hypercuboid face $\bm{n}(\hyperface^{T})$ which contributes to directed-hyperarea vector $\bm{n}_{12}$ (left). The lumped-normal vector $\bm{n}(\hyperface^{T})$ which contributes to $\bm{n}_{13}$ (center). The lumped-normal vector $\bm{n}(\hyperface^{T})$ which contributes to $\bm{n}_{14}$ (right). These lumped-normal vectors are shown for a generic tetrahedron $T$ which belongs to $\mathbb{T}^{(3)}_{12}, \mathbb{T}^{(3)}_{13}$, and $\mathbb{T}^{(3)}_{14}$. The (partial) median-dual region around node $\bm{p}_{1}$ is highlighted in blue.}
    \label{fig:hypercuboidfacet}
\end{figure}

\section{Theoretical Properties in $\mathbb{R}^{\lowercase{d}}$} \label{theoretical_properties_d_section}

In this section, we establish important theoretical properties and conjectures which govern  median-dual regions in $\mathbb{R}^{d}$.

\begin{conjecture}[Directed-Hyperarea Vector Identity in $\mathbb{R}^{d}$]
    Suppose that we construct a median-dual region around an interior point $\bm{p}_{j} \in \mathbb{R}^{d}$. In addition, let a generic adjacent edge be denoted by $\bm{p}_{k} - \bm{p}_{j}$. Under these circumstances, the following identity holds
    \begin{align}
        \bm{n}_{jk} = \frac{2}{d(d+1)} \sum_{T \in\left\{\mathbb{T}^{(d)}_{jk} \right\}} \bm{n}_{j}^{T},
    \label{vector_identity_conjecture}
    \end{align}
    where $\bm{n}_{j}^{T}$ is a normal vector associated with the $(d-1)$-face of $T$ opposite $\bm{p}_j$. We note that $\bm{n}_{j}^{T}$ has a magnitude which is equal to the hypervolume of the opposite face, i.e.~$\left\| \bm{n}_{j}^{T}\right\| = |\mathrm{opp}_{j}(T)|$. 
    \label{directed_hyperarea_conjecture}
\end{conjecture}

\begin{proof}
    We are unaware of a proof for this conjecture for $d > 4$. However, the cases of $d =2$ and 3 were proven in~\cite{Nishikawa:ijnmf2024_inreview}. In addition, the case of $d = 4$ can be proven using careful arguments, see  Theorem~\ref{normal_vector_identity_theorem} of the present paper.
\end{proof}

\begin{theorem}[Hypervolume Identity in $\mathbb{R}^{d}$]
The following relationship between the hypervolume of the median-dual region $V(\bm{p}_{j})$ and the directed-hyperarea vectors $\bm{n}_{jk}$ holds 
\begin{align}
     V(\bm{p}_j) &= \frac{1}{d^2(d+1)} \sum_{k=1}^{M_1} \left( \sum_{T \in\left\{\mathbb{T}^{(d)}_{jk} \right\}} \left(\bm{p}_{k} - \bm{p}_j\right) \cdot  \bm{n}_{j}^{T} \right) \label{volume_identity} \\[1.0ex]
    &=\frac{1}{2d} \sum_{k=1}^{M_1} \left(\bm{p}_{k} - \bm{p}_j\right) \cdot \bm{n}_{jk}.
    \label{volume_identity_prime}
\end{align}
It is important to note that the first line holds in all cases, whereas the second line only holds if we assume that Conjecture~\ref{directed_hyperarea_conjecture} holds.
\label{hypervolume_theorem}
\end{theorem}

\begin{proof}
    Our objective is to prove that Eqs.~\eqref{volume_def} and~\eqref{volume_identity} are equivalent. Towards this end, let us consider the \emph{altitude} $h_{j}^{T}$ of the face opposite $\bm{p}_{j}$ for each $T \in\left\{\mathbb{T}^{(d)}_{jk} \right\}$, which is given~by
    \begin{align}
        h_{j}^{T} \equiv \left(\bm{p}_{k} - \bm{p}_j\right) \cdot  \frac{\bm{n}_{j}^{T}}{\left\| \bm{n}_{j}^{T}\right\|},
        \label{volume_proof_two}
    \end{align}
    where $\left\| \bm{n}_{j}^{T} \right\|$ is the hypervolume of the face opposite $\bm{p}_{j}$.
    Upon substituting Eq.~\eqref{volume_proof_two} into  Eq.~\eqref{volume_identity}, we obtain
    \begin{align}
        V(\bm{p}_j) = \frac{1}{d^2(d+1)} \sum_{k=1}^{M_1} \left( \sum_{T \in\left\{\mathbb{T}^{(d)}_{jk} \right\}} h_{j}^{T} \left\| \bm{n}_{j}^{T} \right\| \right).
        \label{volume_proof_three}
    \end{align}
    Furthermore, the hypervolume of each $d$-simplex $T \in\left\{\mathbb{T}^{(d)}_{jk} \right\}$ is given by
    \begin{align}
        \left| T \right| = \frac{h_{j}^{T} }{d} \left\| \bm{n}_{j}^{T} \right\|.
        \label{volume_proof_four}
    \end{align}
    Substituting Eq.~\eqref{volume_proof_four} into Eq.~\eqref{volume_proof_three} yields
    \begin{align}
        V(\bm{p}_j) = \frac{1}{d(d+1)} \sum_{k=1}^{M_1} \left( \sum_{T \in\left\{\mathbb{T}^{(d)}_{jk} \right\}} \left|T \right| \right). \label{volume_proof_five}
    \end{align}
    Now, we must compare our newly obtained expression, Eq.~\eqref{volume_proof_five}, to our previously obtained expression for the median-dual hypervolume, Eq.~\eqref{volume_def}. If these equations are identical, then we can immediately conclude that our desired result, Eq.~\eqref{volume_identity}, holds. Upon attempting to equate Eqs.~\eqref{volume_proof_five} and \eqref{volume_def}, we require that
    \begin{align}
        \nonumber \frac{1}{d} \sum_{k=1}^{M_1} \left( \sum_{T \in\left\{\mathbb{T}^{(d)}_{jk} \right\}} \left|T \right| \right) &=  \sum_{m=1}^{M_d} \left| T_{j,m}^{(d)} \right| \\[1.0ex]
        &= \sum_{T \in \left\{ \mathbb{T}^{(d)}_{j}\right\}} \left| T \right|.
        \label{volume_proof_six}
    \end{align}
    Careful analysis reveals that the equality in Eq.~\eqref{volume_proof_six} is correct. By inspection, the left hand side of the equation is a weighted summation over the $d$-simplices that share the edges $\bm{p}_{k}-\bm{p}_{j}$, and the right hand side is a summation over the $d$-simplices which share the point $\bm{p}_{j}$. Evidently, the left hand side and right hand side of our equation contain the \emph{same} simplices, as all simplices which share the point $\bm{p}_{j}$ also contain exactly $d$ edges $\bm{p}_{k}-\bm{p}_{j}$. In fact, it is impossible for a $d$-simplex to contain the point $\bm{p}_{j}$, but fail to contain $d$ of the edges $\bm{p}_{k}-\bm{p}_{j}$ which emanate from it. This observation follows from the definition of a $d$-simplex in $\mathbb{R}^{d}$, as each point of the simplex is connected via edges to $d$ other points of the same simplex. In addition, we note that the summation over edges on the left hand side of our equation must be normalized by a factor of $(1/d)$ in order to account for the fact that each simplex hypervolume $|T|$ contributes to the summation $d$ times, as it is shared by $d$ edges which connect to $\bm{p}_{j}$, (as mentioned previously). This final observation completes the proof. 
\end{proof}

\section{Theoretical Properties in $\mathbb{R}^{4}$} \label{theoretical_properties_four_section}

In this section, we establish an important  theoretical property which governs median-dual regions in $\mathbb{R}^{4}$, along with several key implementation and verification details.

\begin{theorem}[Directed-Hyperarea Vector Identity in $\mathbb{R}^{4}$]
    Suppose that we construct a median-dual region around an interior point $\bm{p}_{j} \in \mathbb{R}^{4}$. In addition, let a generic adjacent edge be denoted by $\bm{p}_{k} - \bm{p}_{j}$. Under these circumstances, the following identity holds
    \begin{align}
        \bm{n}_{jk} = \frac{1}{10} \sum_{T \in\left\{\mathbb{T}^{(4)}_{jk} \right\}} \bm{n}_{j}^{T},
        \label{vector_identity}
    \end{align}
    where $\bm{n}_{j}^{T}$ is a normal vector associated with the $3$-face of $T$ opposite $\bm{p}_j$. We note that $\bm{n}_{j}^{T}$ has a magnitude which is equal to the hypervolume of the opposite face, i.e.~$\left\| \bm{n}_{j}^{T}\right\| = |\mathrm{opp}_{j}(T)|$. 
\label{normal_vector_identity_theorem}
\end{theorem}

\begin{proof}
    We start by introducing a generic edge $\bm{p}_{2} - \bm{p}_{1}$ which connects points $\bm{p}_{1} \in \mathbb{R}^4$ and $\bm{p}_{2} \in \mathbb{R}^4$. For the sake of this proof, we are interested in computing the directed-hyperarea vector $\bm{n}_{12}$, and showing that the left hand side of Eq.~\eqref{vector_identity} is equivalent to the right hand side, for this case. Towards this end, we assume that the edge $\bm{p}_{2} - \bm{p}_{1}$ is shared by a total of $N_4$ pentatopes ($4$-simplices), where $N_4 = \mathrm{card}\left(\mathbb{T}_{12}^{(4)} \right)$. Let us consider a single one of these $4$-simplices
    \begin{align}
        T_{\alpha} = \mathrm{conv}\left(\left\{\bm{p}_{1},\bm{p}_{2},\bm{p}_{3},\bm{p}_{4}, \bm{p}_{5} \right\}\right), \qquad T_{\alpha} \in \mathbb{T}_{12}^{(4)}.
    \end{align}
    In a natural fashion, the simplex $T_{\alpha}$ contributes to our calculation of $\bm{n}_{12}$. In particular, in accordance with the definition of our directed-hyperarea  vector, (see Eq.~\eqref{cube_facet_normal}), 
    we find that
    \begin{align}
        \bm{n}_{12} = \sum_{T \in\left\{\mathbb{T}^{(4)}_{12} \right\}} \bm{n} \left( \hyperface^{T} \right), \label{normal_lhs_special}
    \end{align}
    where
    \begin{align}
        \hyperface^{T} = \mathrm{conv}\left( \bm{c} \in \left(\mathcal{C}_{12} \cap T \right) \right).
    \end{align}
    Naturally, the contribution from $T_{\alpha}$ to 
    Eq.~\eqref{normal_lhs_special} appears to be given by $\bm{n}(\hyperface^{T_{\alpha}})$. However, this is not the complete truth, as some of the contributions to $\bm{n}_{12}$ from $T_{\alpha}$ will be canceled by its neighbors. In order to avoid overlooking these cancellations, we must simultaneously compute the contributions to $\bm{n}_{12}$ from $T_{\alpha}$ \emph{and} its $3$-face neighbors. With this in mind, let us define the three, 3-face neighbors of $T_{\alpha}$ as follows
    \begin{align}
        T_{\beta} &= \mathrm{conv}\left(\left\{\bm{p}_{1},\bm{p}_{2},\bm{p}_{3},\bm{p}_{4}, \bm{p}_{6} \right\}\right), \\[1.0ex]
        T_{\gamma} &= \mathrm{conv}\left(\left\{\bm{p}_{1},\bm{p}_{2},\bm{p}_{3},\bm{p}_{5}, \bm{p}_{7} \right\}\right), \\[1.0ex]
        T_{\delta} &= \mathrm{conv}\left(\left\{\bm{p}_{1},\bm{p}_{2},\bm{p}_{4},\bm{p}_{5}, \bm{p}_{8} \right\}\right),
    \end{align}
    where $T_{\beta}, T_{\gamma}, T_{\delta} \in \mathbb{T}_{12}^{(4)}$. Figure~\ref{fig:cluster} shows an illustration of $T_{\alpha}$ and its face neighbors $T_{\beta}, T_{\gamma}$, and~$T_{\delta}$.
    \begin{figure}[h!]
    \centering
    \includegraphics[width = 0.7\textwidth]{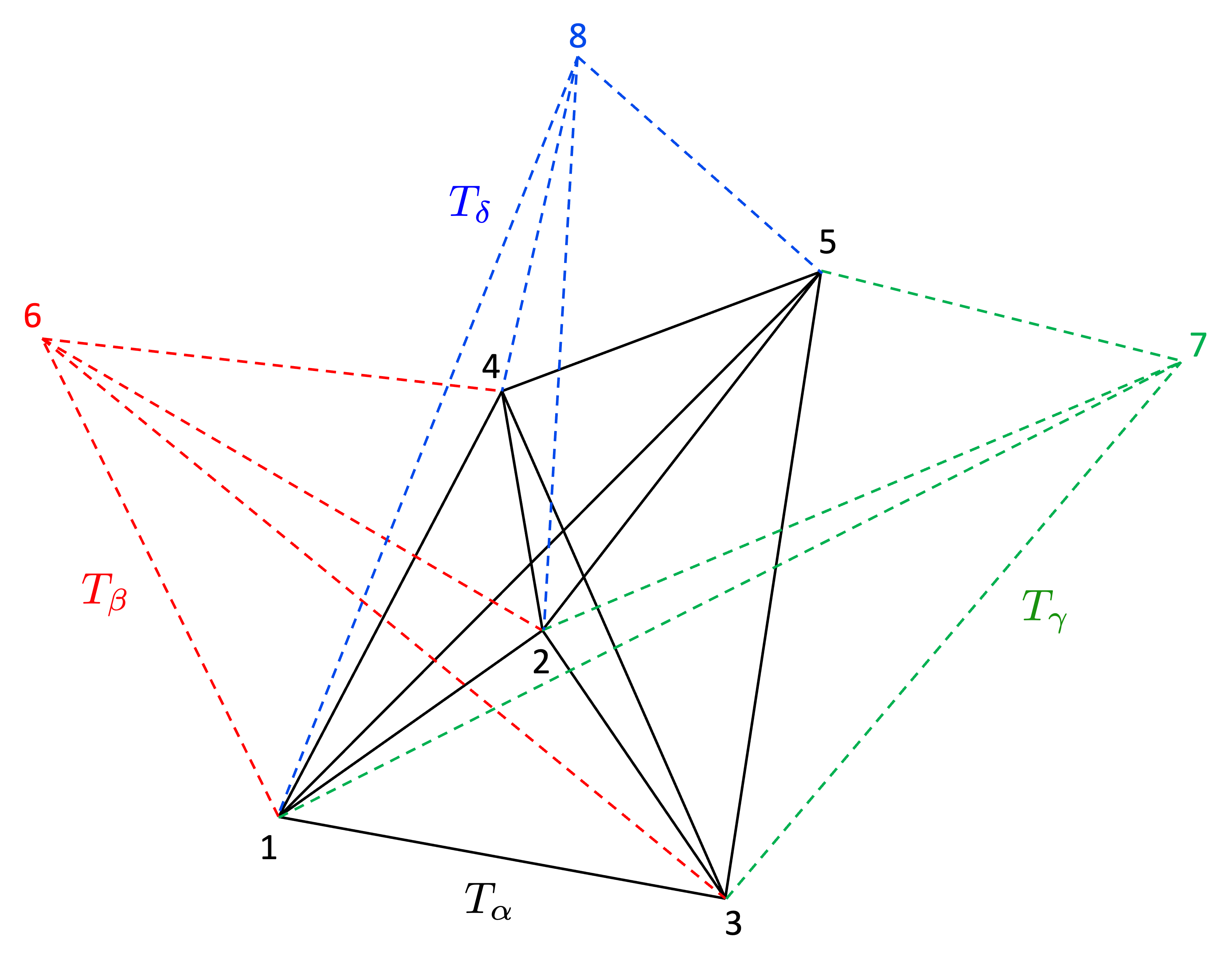}
    \caption{An illustration of the 4-simplex $T_{\alpha}$ and its three face neighbors. The simplex $T_{\alpha}$ is drawn with solid black lines, and its neighboring simplices are drawn with dashed colored lines. In particular, $T_{\beta}$ is drawn with dashed red lines, $T_{\gamma}$ is drawn with dashed green lines, and $T_{\delta}$ is drawn with dashed blue lines.}
    \label{fig:cluster}
    \end{figure}

    We can now compute the individual contributions $\bm{n}\left(\hyperface^{T_{\alpha}}\right)$, $\bm{n}\left(\hyperface^{T_{\beta}}\right)$, $\bm{n}\left(\hyperface^{T_{\gamma}}\right)$, and $\bm{n}\left(\hyperface^{T_{\delta}}\right)$. These lumped-normal vectors are computed based on the median quantities associated with each $d$-simplex. Let us begin by computing the relevant median quantities for $T_{\alpha}$. The centroids $\bm{c} \in \left(\mathcal{C}_{12} \cap T_{\alpha} \right)$ are given by
    \begin{align}
        \bm{c}_{a} &= \frac{1}{2} \left(\bm{p}_{1} + \bm{p}_{2} \right), \\[1.0ex]
        \bm{c}_{b} &= \frac{1}{3} \left(\bm{p}_{1} + \bm{p}_{2} + \bm{p}_{3} \right), \\[1.0ex]
        \bm{c}_{c} &= \frac{1}{3}\left(\bm{p}_{1} + \bm{p}_{2} + \bm{p}_{4} \right), \\[1.0ex]
        \bm{c}_{d} &= \frac{1}{3} \left(\bm{p}_{1} + \bm{p}_{2} + \bm{p}_{5} \right), \\[1.0ex]
        \bm{c}_{e} &= \frac{1}{4} \left(\bm{p}_{1} + \bm{p}_{2} + \bm{p}_{3} + \bm{p}_{4} \right), \\[1.0ex]
        \bm{c}_{f} &= \frac{1}{4}\left(\bm{p}_{1} + \bm{p}_{2} + \bm{p}_{4} + \bm{p}_{5} \right), \\[1.0ex]
        \bm{c}_{g} &= \frac{1}{4} \left(\bm{p}_{1} + \bm{p}_{2} + \bm{p}_{3} + \bm{p}_{5} \right), \\[1.0ex]
        \bm{c}_{h} &= \frac{1}{5} \left(\bm{p}_{1} + \bm{p}_{2} + \bm{p}_{3} + \bm{p}_{4} + \bm{p}_{5} \right).
    \end{align}
    These centroids form a $3$-cuboid with eight vertices. We denote this cuboid by $\hyperface^{T_{\alpha}}$, where 
    \begin{align}
         \hyperface^{T_{\alpha}} = \mathrm{conv}\left( \bm{c} \in \left(\mathcal{C}_{12} \cap T_{\alpha} \right) \right) = \mathrm{conv}\left(\left\{\bm{c}_{a}, \bm{c}_{b}, \bm{c}_{c}, \bm{c}_{d}, \bm{c}_{e}, \bm{c}_{f}, \bm{c}_{g}, \bm{c}_{h} \right\} \right).
    \end{align}
    Computing the lumped-normal vector of the cuboid, $\bm{n}(\hyperface^{T_{\alpha}})$, is relatively straightforward. We simply compute the Coxeter-Freudenthal-Kuhn~\cite{coxeter1934discrete,freudenthal1942simplizialzerlegungen} (CFK) triangulation of the cuboid, and then aggregate (sum over) the normal vectors of the resulting $3$-simplices. These $3$-simplices are given by
    \begin{align}
        T^{(3)}_{A} &= \mathrm{conv}\left(\left\{\bm{c}_{a}, \bm{c}_{b}, \bm{c}_{d}, \bm{c}_{e} \right\} \right), \qquad T^{(3)}_{B} = \mathrm{conv}\left(\left\{\bm{c}_{a}, \bm{c}_{c}, \bm{c}_{d}, \bm{c}_{e} \right\} \right), \\[1.0ex] 
        T^{(3)}_{C} &= \mathrm{conv}\left(\left\{\bm{c}_{c}, \bm{c}_{d}, \bm{c}_{e}, \bm{c}_{f} \right\} \right), \qquad T^{(3)}_{D} = \mathrm{conv}\left(\left\{\bm{c}_{b}, \bm{c}_{d}, \bm{c}_{e}, \bm{c}_{g} \right\} \right), \\[1.0ex] 
        T^{(3)}_{E} &= \mathrm{conv}\left(\left\{\bm{c}_{e}, \bm{c}_{d}, \bm{c}_{g}, \bm{c}_{h} \right\} \right), \qquad T^{(3)}_{F} = \mathrm{conv}\left(\left\{\bm{c}_{d}, \bm{c}_{e}, \bm{c}_{f}, \bm{c}_{h} \right\} \right).
    \end{align}
    Figure~\ref{fig:cfk} shows an illustration of the cuboid, and the associated CFK triangulation.
    \begin{figure}[h!]
    \centering
    \includegraphics[width = 0.8\textwidth]{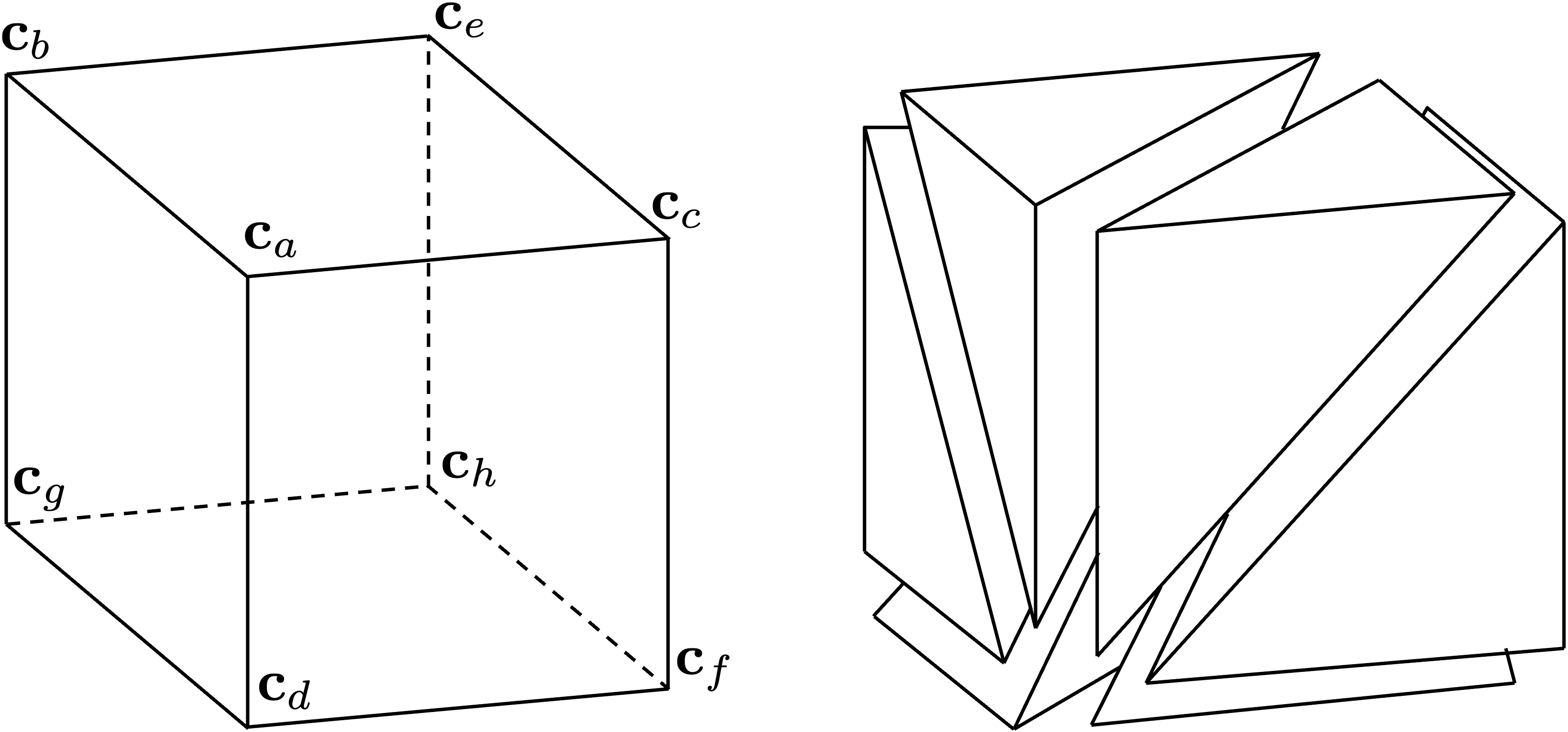}
    \caption{An idealized illustration of the cuboid $\hyperface^{T_{\alpha}}$ (left). The Coxeter-Freudenthal-Kuhn triangulation of the cuboid (right).}
    \label{fig:cfk}
    \end{figure}
    
    The spanning vectors for the $3$-simplices are given by
    \begin{align}
        T^{(3)}_{A}&: \bm{v}_{ba} = \bm{c}_{a} - \bm{c}_{b}, \quad \bm{v}_{bd} = \bm{c}_{d} - \bm{c}_{b}, \quad \bm{v}_{be} = \bm{c}_{e} - \bm{c}_{b}, \\[1.0ex]
        T^{(3)}_{B}&: \bm{v}_{ce} = \bm{c}_{e} - \bm{c}_{c}, \quad \bm{v}_{ca} = \bm{c}_{a} - \bm{c}_{c}, \quad \bm{v}_{cd} = \bm{c}_{d} - \bm{c}_{c}, \\[1.0ex]
        T^{(3)}_{C}&: \bm{v}_{ce} = \bm{c}_{e} - \bm{c}_{c}, \quad \bm{v}_{cd} = \bm{c}_{d} - \bm{c}_{c}, \quad \bm{v}_{cf} = \bm{c}_{f} - \bm{c}_{c}, \\[1.0ex]
        T^{(3)}_{D}&: \bm{v}_{be} = \bm{c}_{e} - \bm{c}_{b}, \quad \bm{v}_{bd} = \bm{c}_{d} - \bm{c}_{b}, \quad \bm{v}_{bg} = \bm{c}_{g} - \bm{c}_{b}, \\[1.0ex]
        T^{(3)}_{E}&: \bm{v}_{gh} = \bm{c}_{h} - \bm{c}_{g}, \quad \bm{v}_{ge} = \bm{c}_{e} - \bm{c}_{g}, \quad \bm{v}_{gd} = \bm{c}_{d} - \bm{c}_{g}, \\[1.0ex]
        T^{(3)}_{F}&: \bm{v}_{fh} = \bm{c}_{h} - \bm{c}_{f}, \quad \bm{v}_{fd} = \bm{c}_{d} - \bm{c}_{f}, \quad \bm{v}_{fe} = \bm{c}_{e} - \bm{c}_{f}.
    \end{align}
    The associated normal vectors are
    \begin{align}
        \label{volume_one}
        T^{(3)}_{A}&: \frac{1}{6} \bm{v}_{ba} \times \bm{v}_{bd} \times \bm{v}_{be}, \qquad T^{(3)}_{B}: \frac{1}{6} \bm{v}_{ce} \times \bm{v}_{ca} \times \bm{v}_{cd}, \\[1.0ex]
        T^{(3)}_{C}&: \frac{1}{6} \bm{v}_{ce} \times \bm{v}_{cd} \times \bm{v}_{cf}, \qquad T^{(3)}_{D}: \frac{1}{6} \bm{v}_{be} \times \bm{v}_{bd} \times \bm{v}_{bg}, \\[1.0ex]
        T^{(3)}_{E}&: \frac{1}{6} \bm{v}_{gh} \times \bm{v}_{ge} \times \bm{v}_{gd}, \qquad T^{(3)}_{F}: \frac{1}{6} \bm{v}_{fh} \times \bm{v}_{fd} \times \bm{v}_{fe}.
        \label{volume_six}
    \end{align}
    Here, we have assumed a positive orientation for the normal vector associated with $T_{A}^{(3)}$, and then insisted that all other normal vectors (for $T_{B}^{(3)}$--$T_{F}^{(3)}$) are consistent with this orientation. Note that, we say a vector $\bm{n} \in \mathbb{R}^4$ has a positive orientation if $(\bm{p}_{2} - \bm{p}_{1}) \cdot \bm{n} > 0$. Furthermore, we note that the generalized cross product in $\mathbb{R}^4$ is defined as follows
    \begin{align}
        \bm{u} \times \bm{v} \times \bm{w} = \bm{e}_{1} \begin{vmatrix}
            u_2 & u_3 & u_4 \\[1.0ex]
            v_2 & v_3 & v_4 \\[1.0ex]
            w_2 & w_3 & w_4 
        \end{vmatrix} - \bm{e}_{2} \begin{vmatrix}
            u_1 & u_3 & u_4 \\[1.0ex]
            v_1 & v_3 & v_4 \\[1.0ex]
            w_1 & w_3 & w_4 
        \end{vmatrix} + \bm{e}_{3} \begin{vmatrix}
            u_1 & u_2 & u_4 \\[1.0ex]
            v_1 & v_2 & v_4 \\[1.0ex]
            w_1 & w_2 & w_4 
        \end{vmatrix} - \bm{e}_{4} \begin{vmatrix}
            u_1 & u_2 & u_3 \\[1.0ex]
            v_1 & v_2 & v_3 \\[1.0ex]
            w_1 & w_2 & w_3 
        \end{vmatrix},
    \end{align}
    where $\bm{u}, \bm{v}, \bm{w} \in \mathbb{R}^4$, $\bm{e}_i \in \mathbb{R}^4$, and $i = 1, \ldots, 4$ is the set of coordinate vectors which are 1 in the $i$-th coordinate direction and 0 in all other directions. 
    
    Next, we can sum over the normal vectors (Eqs.~\eqref{volume_one}--\eqref{volume_six}) in order to construct $\bm{n}(\hyperface^{T_{\alpha}})$, as follows
    \begin{align}
        \nonumber \bm{n}(\hyperface^{T_{\alpha}}) = \frac{1}{6} &\left[\bm{v}_{ba} \times \bm{v}_{bd} \times \bm{v}_{be} + \bm{v}_{ce} \times \bm{v}_{ca} \times \bm{v}_{cd} + \bm{v}_{ce} \times \bm{v}_{cd} \times \bm{v}_{cf} \right. \\[1.0ex]
        \nonumber &+ \left. \bm{v}_{be} \times \bm{v}_{bd} \times \bm{v}_{bg} + \bm{v}_{gh} \times \bm{v}_{ge} \times \bm{v}_{gd} + \bm{v}_{fh} \times \bm{v}_{fd} \times \bm{v}_{fe}\right] \\[1.0ex]
        % \nonumber =\frac{1}{6} &\Bigg[ \frac{1}{20}\big( -\bm{p}_{1} \times \bm{p}_{4} \times \bm{p}_{5} - \bm{p}_{2} \times \bm{p}_{4} \times \bm{p}_{5} + \bm{p}_{1} \times \bm{p}_{3} \times \bm{p}_{5}   \\[1.0ex]
        %  \nonumber &+ \bm{p}_{2} \times \bm{p}_{3} \times \bm{p}_{5} - \bm{p}_{1}  \times \bm{p}_{3} \times \bm{p}_{4} - \bm{p}_{2} \times \bm{p}_{3} \times \bm{p}_{4} \big) + \frac{1}{10} \bm{p}_{3} \times \bm{p}_{4} \times \bm{p}_{5} \Bigg] \\[1.0ex]
        \nonumber =\frac{1}{60} &\Bigg[ \frac{1}{2}\big( -\bm{p}_{1} \times \bm{p}_{4} \times \bm{p}_{5} - \bm{p}_{2} \times \bm{p}_{4} \times \bm{p}_{5} + \bm{p}_{1} \times \bm{p}_{3} \times \bm{p}_{5}   \\[1.0ex]
         &+ \bm{p}_{2} \times \bm{p}_{3} \times \bm{p}_{5} - \bm{p}_{1}  \times \bm{p}_{3} \times \bm{p}_{4} - \bm{p}_{2} \times \bm{p}_{3} \times \bm{p}_{4} \big) +  \bm{p}_{3} \times \bm{p}_{4} \times \bm{p}_{5} \Bigg]. \label{normal_alpha_A}
    \end{align}
    Thereafter, we can substitute $\bm{p}_{1} = \bm{p}_{2} - (\bm{p}_{2} - \bm{p}_{1})$ into Eq.~\eqref{normal_alpha_A}, such that
    \begin{align}
        \nonumber \bm{n}(\hyperface^{T_{\alpha}}) = \frac{1}{60} &\Bigg[ \bm{p}_{3} \times \bm{p}_{4} \times \bm{p}_{5} + \bm{p}_{2} \times \bm{p}_{3} \times \bm{p}_{5} - \bm{p}_{2} \times \bm{p}_{3} \times \bm{p}_{4} - \bm{p}_{2} \times \bm{p}_{4} \times \bm{p}_{5} \\[1.0ex]
        &+ \frac{1}{2} \left( \left(\bm{p}_{2} - \bm{p}_{1} \right) \times \bm{p}_{4} \times \bm{p}_{5} - \left(\bm{p}_{2} - \bm{p}_{1} \right) \times \bm{p}_{3} \times \bm{p}_{5} + \left(\bm{p}_{2} - \bm{p}_{1} \right) \times \bm{p}_{3} \times \bm{p}_{4} \right) \Bigg].
        \label{normal_intermediate}
    \end{align}
    The quantity on the first line of Eq.~\eqref{normal_intermediate} has special significance. In order to see this, consider the face opposite $\bm{p}_{1}$ in the element $T_{\alpha}$
    \begin{align}
        \mathrm{opp}_{1}(T_{\alpha}): &  \; \mathrm{conv}\left(\left\{\bm{p}_{2}, \bm{p}_{3},\bm{p}_{4}, \bm{p}_{5} \right\} \right).
    \end{align}
    The spanning vectors for this face are
    \begin{align}
           \mathrm{opp}_{1}(T_{\alpha}): &  \; \bm{v}_{23} = \bm{p}_{3} - \bm{p}_{2}, \quad \bm{v}_{24} = \bm{p}_{4} - \bm{p}_{2}, \quad \bm{v}_{25} = \bm{p}_{5} - \bm{p}_{2}.
    \end{align}
    In turn, the associated normal vector is
    \begin{align}
        \bm{n}_{1}^{T_{\alpha}} &= \frac{1}{6} \bm{v}_{23} \times \bm{v}_{24} \times \bm{v}_{25},
    \end{align}
    or equivalently,
    \begin{align}
        \label{opposite_normal_one}
        \bm{n}_{1}^{T_{\alpha}} &= \frac{1}{6} \left(\bm{p}_{3} \times \bm{p}_{4} \times \bm{p}_{5} + \bm{p}_{2} \times \bm{p}_{3} \times \bm{p}_{5} - \bm{p}_{2} \times \bm{p}_{3} \times \bm{p}_{4} - \bm{p}_{2} \times \bm{p}_{4} \times \bm{p}_{5} \right).
    \end{align}
    Here, we assume that $\bm{n}_{1}^{T_{\alpha}}$ has a positive orientation. Upon substituting Eq.~\eqref{opposite_normal_one} into Eq.~\eqref{normal_intermediate}, we obtain
    \begin{align}
        \bm{n}(\hyperface^{T_{\alpha}}) = \frac{1}{10} \Bigg[ \bm{n}_{1}^{T_{\alpha}} + \frac{1}{12} \left(\bm{p}_{2} - \bm{p}_{1} \right) \times \left(  \bm{p}_{4} \times \bm{p}_{5} - \bm{p}_{3} \times \bm{p}_{5} + \bm{p}_{3} \times \bm{p}_{4} \right) \Bigg]. \label{normal_alpha_B}
    \end{align}
    Setting Eq.~\eqref{normal_alpha_B} aside for the moment, we can introduce the following identities
    \begin{align}
        \nonumber &\left(\bm{p}_{2} - \bm{p}_{1} \right) \times \bm{p}_{3} \times \bm{p}_{4} \\[1.0ex]
        \nonumber &= \left(\bm{p}_{2} - \bm{p}_{1} \right) \times \left[ \left(\bm{p}_{3} -\bm{p}_{1} + \bm{p}_{1} \right) \times \left(\bm{p}_{4} -\bm{p}_{1} + \bm{p}_{1} \right) \right] \\[1.0ex]
        &= \left(\bm{p}_{2} - \bm{p}_{1} \right) \times \left[ \left(\bm{p}_{3} -\bm{p}_{1} \right) \times \left(\bm{p}_{4} -\bm{p}_{1} \right) + \left(\bm{p}_{3} -\bm{p}_{1} \right) \times \bm{p}_{1} + \bm{p}_{1} \times \left(\bm{p}_{4} - \bm{p}_{1} \right)\right], \label{cross_id_one} \\[1.0ex]
        \nonumber &\left(\bm{p}_{2} - \bm{p}_{1} \right) \times \bm{p}_{3} \times \bm{p}_{5} \\[1.0ex]
        \nonumber &= \left(\bm{p}_{2} - \bm{p}_{1} \right) \times \left[ \left(\bm{p}_{3} -\bm{p}_{1} + \bm{p}_{1} \right) \times \left(\bm{p}_{5} -\bm{p}_{1} + \bm{p}_{1} \right) \right] \\[1.0ex]
        &= \left(\bm{p}_{2} - \bm{p}_{1} \right) \times \left[ \left(\bm{p}_{3} -\bm{p}_{1} \right) \times \left(\bm{p}_{5} -\bm{p}_{1} \right) + \left(\bm{p}_{3} -\bm{p}_{1} \right) \times \bm{p}_{1} + \bm{p}_{1} \times \left(\bm{p}_{5} - \bm{p}_{1} \right)\right], \label{cross_id_two} \\[1.0ex]
        \nonumber &\left(\bm{p}_{2} - \bm{p}_{1} \right) \times \bm{p}_{4} \times \bm{p}_{5} \\[1.0ex]
        \nonumber &= \left(\bm{p}_{2} - \bm{p}_{1} \right) \times \left[ \left(\bm{p}_{4} -\bm{p}_{1} + \bm{p}_{1} \right) \times \left(\bm{p}_{5} -\bm{p}_{1} + \bm{p}_{1} \right) \right] \\[1.0ex]
        &= \left(\bm{p}_{2} - \bm{p}_{1} \right) \times \left[ \left(\bm{p}_{4} -\bm{p}_{1} \right) \times \left(\bm{p}_{5} -\bm{p}_{1} \right) + \left(\bm{p}_{4} -\bm{p}_{1} \right) \times \bm{p}_{1} + \bm{p}_{1} \times \left(\bm{p}_{5} - \bm{p}_{1} \right)\right]. \label{cross_id_three}
    \end{align}
    Upon combining Eqs.~\eqref{cross_id_one}--\eqref{cross_id_three}, one obtains
    \begin{align}
        \nonumber &\left(\bm{p}_{2} - \bm{p}_{1} \right) \times \left(  \bm{p}_{4} \times \bm{p}_{5} - \bm{p}_{3} \times \bm{p}_{5} + \bm{p}_{3} \times \bm{p}_{4} \right) \\[1.0ex]
        &= \left(\bm{p}_{2} - \bm{p}_{1} \right) \times \Big[ \left(\bm{p}_{4} -\bm{p}_{1} \right) \times \left(\bm{p}_{5} -\bm{p}_{1} \right) - \left(\bm{p}_{3} -\bm{p}_{1} \right) \times \left(\bm{p}_{5} -\bm{p}_{1} \right) +\left(\bm{p}_{3} -\bm{p}_{1} \right) \times \left(\bm{p}_{4} -\bm{p}_{1} \right) \Big]. \label{cross_id_comb}
    \end{align}
    We now return our attention to Eq.~\eqref{normal_alpha_B}. Upon substituting Eq.~\eqref{cross_id_comb} into Eq.~\eqref{normal_alpha_B}, one obtains
    \begin{align}
        \nonumber \bm{n}(\hyperface^{T_{\alpha}}) = \frac{1}{10} \Bigg[ \bm{n}_{1}^{T_{\alpha}} + \frac{1}{12} &\Big[ \left(\bm{p}_{2} - \bm{p}_{1} \right) \times \left(\bm{p}_{4} -\bm{p}_{1} \right) \times \left(\bm{p}_{5} -\bm{p}_{1} \right)  \\[1.0ex]
        \nonumber &-\left(\bm{p}_{2} - \bm{p}_{1} \right) \times \left(\bm{p}_{3} -\bm{p}_{1} \right) \times \left(\bm{p}_{5} -\bm{p}_{1} \right) \\[1.0ex]
         &+\left(\bm{p}_{2} - \bm{p}_{1} \right) \times \left(\bm{p}_{3} -\bm{p}_{1} \right) \times \left(\bm{p}_{4} -\bm{p}_{1} \right) \Big] \Bigg]. \label{normal_alpha_C}
    \end{align}
    The rightmost terms inside the parentheses of Eq.~\eqref{normal_alpha_C} are associated with the opposite faces of $T_{\alpha}$. In particular, we can introduce three such opposite faces
    \begin{align}
        \mathrm{opp}_{3}(T_{\alpha}): &  \; \mathrm{conv}\left(\left\{\bm{p}_{1}, \bm{p}_{2},\bm{p}_{4}, \bm{p}_{5} \right\} \right), \\[1.0ex]
        \mathrm{opp}_{4}(T_{\alpha}): &  \; \mathrm{conv}\left(\left\{\bm{p}_{1}, \bm{p}_{2},\bm{p}_{3}, \bm{p}_{5} \right\} \right), \\[1.0ex]
        \mathrm{opp}_{5}(T_{\alpha}): &  \; \mathrm{conv}\left(\left\{\bm{p}_{1}, \bm{p}_{2},\bm{p}_{3}, \bm{p}_{4} \right\} \right).
    \end{align}
    The associated normal vectors are given by
    \begin{align}
        \bm{n}_{3}^{T_{\alpha}} &= \frac{1}{6} \left(\bm{p}_{2} - \bm{p}_{1} \right) \times \left(\bm{p}_{4} -\bm{p}_{1} \right) \times \left(\bm{p}_{5} -\bm{p}_{1} \right), \label{normal_opp_three}  \\[1.0ex]
        \bm{n}_{4}^{T_{\alpha}} &= \frac{1}{6} \left(\bm{p}_{2} - \bm{p}_{1} \right) \times \left(\bm{p}_{5} -\bm{p}_{1} \right) \times \left(\bm{p}_{3} -\bm{p}_{1} \right), \label{normal_opp_four} \\[1.0ex]
        \bm{n}_{5}^{T_{\alpha}} &= \frac{1}{6} \left(\bm{p}_{2} - \bm{p}_{1} \right) \times \left(\bm{p}_{3} -\bm{p}_{1} \right) \times \left(\bm{p}_{4} -\bm{p}_{1} \right), \label{normal_opp_five}
    \end{align}
    where the orientations of these vectors are taken to be consistent with $\bm{n}_{1}^{T_{\alpha}}$. Next, upon substituting Eqs.~\eqref{normal_opp_three}--\eqref{normal_opp_five} into Eq.~\eqref{normal_alpha_C}, we obtain
    \begin{align}
        \bm{n}(\hyperface^{T_{\alpha}}) = \frac{1}{10} \left[\bm{n}_{1}^{T_{\alpha}} + \frac{1}{2} \left(\bm{n}_{3}^{T_{\alpha}} +\bm{n}_{4}^{T_{\alpha}} + \bm{n}_{5}^{T_{\alpha}} \right) \right].
        \label{normal_alpha_expanded_alternative}
    \end{align}
    Thereafter, we can show that the contributions from $\bm{n}_{3}^{T_{\alpha}}, \bm{n}_{4}^{T_{\alpha}}$, and $\bm{n}_{5}^{T_{\alpha}}$ in Eq.~\eqref{normal_alpha_expanded_alternative} are canceled by neighboring pentatopes $T_{\beta}$, $T_{\gamma}$, and $T_{\delta}$. In particular, we find that
    \begin{align}
         \bm{n}(\hyperface^{T_{\beta}}) = \frac{1}{10} \left[\bm{n}_{1}^{T_{\beta}} + \frac{1}{2} \left(\bm{n}_{3}^{T_{\beta}} +\bm{n}_{4}^{T_{\beta}} + \bm{n}_{6}^{T_{\beta}} \right) \right], \\[1.0ex]
         \bm{n}(\hyperface^{T_{\gamma}}) = \frac{1}{10} \left[\bm{n}_{1}^{T_{\gamma}} + \frac{1}{2} \left(\bm{n}_{3}^{T_{\gamma}} +\bm{n}_{5}^{T_{\gamma}} + \bm{n}_{7}^{T_{\gamma}} \right) \right], \\[1.0ex]
         \bm{n}(\hyperface^{T_{\delta}}) = \frac{1}{10} \left[\bm{n}_{1}^{T_{\delta}} + \frac{1}{2} \left(\bm{n}_{4}^{T_{\delta}} +\bm{n}_{5}^{T_{\delta}} + \bm{n}_{8}^{T_{\delta}} \right) \right],
    \end{align}
    and furthermore
    \begin{align}
        \bm{n}_{3}^{T_{\alpha}} = -\bm{n}_{8}^{T_{\delta}}, \quad \bm{n}_{4}^{T_{\alpha}} = -\bm{n}_{7}^{T_{\gamma}}, \quad \bm{n}_{5}^{T_{\alpha}} = -\bm{n}_{6}^{T_{\beta}}.
    \end{align}
    Therefore, the only net contribution from $T_{\alpha}$ to the summation
    \begin{align}
        \bm{n}(\hyperface^{T_{\alpha}}) + \bm{n}(\hyperface^{T_{\beta}}) + \bm{n}(\hyperface^{T_{\gamma}}) + \bm{n}(\hyperface^{T_{\delta}}),
    \end{align}
    is
    \begin{align}
        \frac{1}{10} \bm{n}_{1}^{T_{\alpha}}.  
    \end{align}
    Since $T_{\alpha} = \mathrm{conv}\left(\left\{\bm{p}_{1},\bm{p}_{2},\bm{p}_{3},\bm{p}_{4}, \bm{p}_{5} \right\}\right)$ is a generic $4$-simplex which shares the edge $\bm{p}_{2} - \bm{p}_{1}$, we conclude that $T_{\beta}$, $T_{\gamma}$, and $T_{\delta}$
    make similar net contributions
    \begin{align}
        \frac{1}{10} \bm{n}_{1}^{T_{\beta}}, \frac{1}{10} \bm{n}_{1}^{T_{\gamma}}, \frac{1}{10} \bm{n}_{1}^{T_{\delta}}.
    \end{align}
    Summing over the contributions from each pentatope that shares the edge $\bm{p}_{2} - \bm{p}_{1}$ produces the desired result (Eq.~\eqref{vector_identity}).
\end{proof}

\begin{remark}[Boundary Considerations]  Theorem~\ref{normal_vector_identity_theorem} needs to be modified slightly if $\bm{p}_{j}$ is a boundary point. In this case, Eq.~\eqref{vector_identity} can be changed as follows
    \begin{align}
        \bm{n}_{jk} = \frac{1}{10} \left( \sum_{T \in\left\{\mathbb{T}^{(4)}_{jk} \right\}} \bm{n}_{j}^{T} + \frac{1}{2} \sum_{B \in \{ \mathbb{T}_{jk}^{(3)} \cap \partial \Omega \}}  \bm{n}_{B}\right),
        \label{vector_identity_boundary}
    \end{align}
    where $\bm{n}_{B}$ is the outward-pointing normal vector of the boundary 3-simplex $B$.
    \label{boundary_remark_01}
        % Explain Eq.~\eqref{normal_alpha_expanded_alternative} needs to be corrected if the edge is a boundary edge, here: simply add $ \sum_{B \in \{ \mathbb{T}_{jk}^{(3)} \cap \partial \Omega \}} \frac{1}{20} {\bf n}_{B}$ to Eq.~\eqref{vector_identity}, where ${\bf n}_{B}$ is the outward hyperarea vector of a boundary $3$-simplex $B$.
\end{remark}

\begin{remark}[Implementation Details]
    We recommend implementing Eq.~\eqref{vector_identity} from Theorem~\ref{normal_vector_identity_theorem} and Eq.~\eqref{vector_identity_boundary} from Remark~\ref{boundary_remark_01}  as follows:
    \begin{itemize}
        \item Loop over each 4-simplex in the triangulation.
        \item For each 4-simplex $T$, loop over the 3-faces. 
        \item For each face, identify the vertices of the face in the global numbering system, e.g.~$\bm{p}_{i_1}, \bm{p}_{i_2}, \bm{p}_{i_3}$, and $\bm{p}_{i_4}$, along with the vertex opposite the face, e.g.~$\bm{p}_{i_5}$.
        \item Update the directed-hyperarea vector for the edge $\bm{p}_{k} - \bm{p}_{j}$ using each vertex of the face
        \begin{align}
            \bm{n}_{jk} = \bm{n}_{jk} + \frac{1}{10} \bm{n}_{j}^{T}, 
        \end{align}
        where $(j, k) \in \{ (r,s) \in \{ i_1, i_2, i_3, i_4, i_5 \}^2 \,\, | \,\, r < s \}$.
        %$j = i_5$ and $k = i_1, i_2, i_3$ and $i_4$. 
        Note that we store only one directed-hyperarea vector per edge, $\bm{n}_{jk}$ in the direction from the smaller node number to the larger node number (i.e. $ j< k$), because its counterpart $\bm{n}_{kj}$ is given by $\bm{n}_{kj}=-\bm{n}_{jk}$.
        \item Loop over each boundary 3-simplex $B$.
        \item For each boundary face, identify the vertices of the face in the global numbering system, e.g.~$\bm{p}_{i_1}, \bm{p}_{i_2}, \bm{p}_{i_3}$ and $\bm{p}_{i_4}$.
        \item Correct the directed-hyperarea vector for the edge $\bm{p}_{k} - \bm{p}_{j}$ using each vertex of the boundary face
        \begin{align}
            \bm{n}_{jk} = \bm{n}_{jk} + \frac{1}{20} \bm{n}_{B},
        \end{align}
        where $(j, k) \in \{ (r,s) \in \{ i_1, i_2, i_3, i_4\}^2 \,\, | \,\, r < s \}$.
      %  \item If necessary, correct the orientation of each directed-hyperarea vector. In particular, if $j < k$, we require that $\bm{n}_{jk} \cdot \left(\bm{p}_{k} - \bm{p}_{j}\right) > 0$.
    \end{itemize}
    % Explain that the sum of the hyperarea vectors is zero at all interior nodes (conservation at interior nodes), and the same is true for boundary nodes if we add $ \frac{1}{5} {\bf n}_{B}$ to each boundary node of a boundary $3$-simplex $B$ (conservation at boundary nodes). {\color{red} [Comment (you can delete this once you read it): a boundary closure formula for 2nd/3rd-order accuracy will have to be derived separately. This is beyond the scope of this paper. See the last sentence added to the first paragraph in the next section.] }
    \label{boundary_remark_02}
\end{remark}

\begin{remark}[Verification Details]
    Once the directed-hyperarea vectors have been computed using the procedure of Remark~\ref{boundary_remark_02}, they can be verified as follows:
    \begin{itemize}
        \item Create node vectors $\bm{a}_{j}$ associated with each node $\bm{p}_j$.
        \item Loop over each edge $\bm{p}_{k} - \bm{p}_{j}$ 
($j < k$) of the triangulation.
        \item Update the node vectors
        \begin{align}
            \bm{a}_{j} = \bm{a}_{j} + \bm{n}_{jk}, \qquad \bm{a}_{k}  = \bm{a}_{k}  - \bm{n}_{jk} \,\, = \bm{a}_{k} + \bm{n}_{kj}.
        \end{align}
        \item Loop over each boundary 3-simplex $B$.
        \item For each boundary face, identify the vertices of the face in the global numbering system, e.g.~$\bm{p}_{i_1}, \bm{p}_{i_2}, \bm{p}_{i_3}$ and $\bm{p}_{i_4}$.
        \item Update the node vector
        \begin{align}
            \bm{a}_{j} = \bm{a}_{j} + \frac{1}{4} \bm{n}_{B},
        \end{align}
        where $j = i_1, i_2, i_3$ and $i_4$.
    \end{itemize}
    When we complete the procedure above, if each $\bm{a}_{j} = \bm{0}$, then we have established the correctness of the implementation.
    \label{boundary_remark_03}
\end{remark}

Similar procedures can be used to verify the implementation of directed hyperarea vectors in dimensions 2 and 3. The interested reader is encouraged to consult~\cite{Nishikawa:ijnmf2024_inreview} for details. 

%%%%%%%%%%%%%%%%%%%%%%%%%%%%%%%%%%%%%%%%%%%%%%%%

 \section{Solver Details and Complexity} \label{complexity_section}

 In what follows, we describe additional details of the edge-based scheme for application (programming) purposes. Thereafter, we perform a complexity analysis in order to characterize the scheme's computational advantages.

\vspace{0.15in}

\subsection{Solver Details}

In this section, we describe the procedure we use to compute the residual for our edge-based, discretization method. 
In addition, we discuss the 2-stage Runge Kutta (RK) solver that we use to converge our method to a final solution. 

Throughout the remainder of this discussion, we assume that our edge-based method is used to solve a scalar partial differential equation of the form
\begin{align}
    \bm{a} \cdot \nabla u = \nabla \cdot (\bm{a} u) = f,
\end{align}
where $\nabla = \left(\partial/\partial x_{1}, \partial/\partial x_{2}, \ldots, \partial/\partial x_{d} \right)$ is the $d$-vector gradient operator, $u$ is the scalar solution, $f$ is a scalar forcing function, and $\bm{a} \in \mathbb{R}^d$ is a constant advection direction. Evidently, 
\begin{align}
     \nabla \cdot \bm{\Phi} = f,
    \label{linear_advection_eq}
\end{align}
where
\begin{align}
    \bm{\Phi} = \bm{a} u, 
\end{align}
is the advective flux. We note that the divergence of the flux can be computed as follows
\begin{align}
    \nabla \cdot \bm{\Phi} = \lim_{|\Omega| \rightarrow 0} \frac{1}{|\Omega|} \int_{\partial \Omega} \bm{\Phi} \cdot \bm{n} \, ds,
\end{align}
where $ds$ is an infinitesimal hyperarea. In a natural fashion, we can approximate the divergence as 
\begin{align}
    \nabla \cdot \bm{\Phi} \approx \frac{1}{V} \int_{\partial \Omega} \bm{\Phi} \cdot \bm{n} \, ds,
\end{align}
for non-zero $V = |\Omega|$. Note that this relation is exact if the flux is a linear function \cite{nishikawa_centroid:JCP2020}, and therefore, the right hand side is a sufficiently accurate approximation to the flux divergence for second-order accuracy. If we choose $\Omega$ to be the median dual volume that contains node $\bm{p}_j$, then the divergence at node $\bm{p}_j$ is 
\begin{align}
    \left(\nabla \cdot \bm{\Phi}\right)_{j} \approx \frac{1}{V_j} \int_{\partial \Omega_j} \bm{\Phi} \cdot \bm{n} \, ds,
\end{align}
where $V_j = V(\bm{p}_j)$ is the dual volume.
%In order to compute the residual at the node, the flux contributions between that node and its neighbors are combined with the contribution of the forcing function as follows
%HERE, I would write as below. What do you think?
In our edge-based method, we discretize the surface integral using numerical fluxes defined at edge midpoints with directed-hyperarea vectors, and construct the residual, a discrete approximation to the linear advection equation (Eq.~\eqref{linear_advection_eq}) at node $\bm{p}_j$, as follows
\begin{align}
    \text{Res}_j = \sum_{\bm{p}_k\in \mathbb{T}_{j}} \Phi_{jk}
    \|\bm{n}_{jk}\| 
    -f_jV_j,
    \label{residual}
\end{align}
where $f_j = f(\bm{p}_j)$ is the forcing function evaluated at node $\bm{p}_j$, $\Phi_{jk}$ is the numerical flux from node $\bm{p}_j$ to node $\bm{p}_k$, and $\mathbb{T}_j$ is the set of all simplices which contain node $\bm{p}_j$. The residual is conservative in the sense that the sum of $Res_j$ over all nodes will depend only on the boundary data. 
The flux $\Phi_{jk}$ incorporates the current solution at both nodes $\bm{p}_j$ and $\bm{p}_k$ through two terms $u_L$ and $u_R$, where $u_L$ involves the gradient at $\bm{p}_j$ while $u_R$ involves the gradient at $\bm{p}_k$ as follows
\begin{align}
    u_L &= u_j + \frac{1}{2} \nabla u_j \cdot (\bm{x}_k-\bm{x}_j), \\
    u_R &= u_k - \frac{1}{2} \nabla u_k \cdot (\bm{x}_k-\bm{x}_j).
\end{align}
The gradients $\nabla u_j$ and $\nabla u_k$ are computed by a linear least-squares (LSQ) method: e.g., at the node $\bm{p}_j$, 
\begin{align}
 \nabla u_j
 =
 \sum_{\bm{p}_k\in \mathbb{T}_{j}} 
 {\bf C}_{jk}
 (u_k - u_j ),
\end{align}
where ${\bf C}_{jk}$ is a vector of LSQ coefficients corresponding to the $k$-th column of $\left( {\bf A}^t {\bf A} \right)^{-1} {\bf A}^t$ (the superscript $t$ indicates transpose), and ${\bf A}^t$ is a $d \times M_1$ matrix defined by 
\begin{align}
  {\bf A}^t
  =
  \left[  \begin{array}{ccccc} 
  \bm{x}_{1}-\bm{x}_j, \,\,\,
  \bm{x}_{2}-\bm{x}_j, \,\,\,
  \bm{x}_{3}-\bm{x}_j, \,\,\, 
  \cdots, \bm{x}_{k}-\bm{x}_j, 
\cdots, \,\,\,
  \bm{x}_{M_1}-\bm{x}_j 
              \end{array} \right].
\end{align}
Here, ${M_1}$ is the number of edge neighbors of $\bm{p}_j$. 
The numerical flux itself, $\Phi_{jk}$, is calculated by averaging the normal components of the flux through the faces of the median dual volume and including an upwinding term that biases the flux in the upstream direction
\begin{align}
    \Phi_{jk} &= \frac{1}{2}( \hat{\bm{n}}_{jk} \cdot \bm{a})(u_L+u_R) - \frac{1}{2}\left| \hat{\bm{n}}_{jk} \cdot \bm{a}\right|(u_L-u_R).
\end{align}
Here, $\hat{\bm{n}}_{jk} = {\bm{n}}_{jk} /  \|\bm{n}_{jk}\|$, and $\hat{\bm{n}}_{jk}$ is oriented from node $\bm{p}_j$ to node $\bm{p}_k$, so that $\Phi_{jk}$ as defined in the above corresponds to the numerical flux out of the median dual volume surrounding node $\bm{p}_j$. 
%$\Phi_{jk}$ is oriented from node $\bm{p}_j$ to node $\bm{p}_k$, so that  positive values of $\Phi_{jk}$ correspond to the numerical flux out of the median dual volume surrounding node $\bm{p}_j$. 
For more complex governing equations, one can define the numerical flux separately in space and time, and then add the resulting contributions  together \cite{Nishikawa:ICCFD2024}. 
 
To accumulate the total flux into or out of a node, we loop through the edges in the mesh and compute the flux between the two nodes that bound each edge. The residuals for these two nodes, $\bm{p}_j$ and $\bm{p}_k$, are updated after each pass in the loop as follows
\begin{align}
    \text{Res}_j &= \text{Res}_j + \Phi_{jk}, \\
    \text{Res}_k &= \text{Res}_k - \Phi_{jk}.
\end{align}
After all numerical fluxes are computed, we loop through the nodes to accumulate the forcing-function contributions into the residuals
\begin{align}
    \text{Res}_j = \text{Res}_j-f_jV_j.
\end{align}
This completes the description of the residual (Eq.~\eqref{residual}). For linear polynomial solutions, it exactly reproduces the advection equation (Eq.~\eqref{linear_advection_eq}) at node $\bm{p}_j$ on arbitrary simplex-element grids, and therefore, we expect the numerical solution to be second-order accurate. 

Once the residuals are fully defined at each node, we proceed with a two-stage RK scheme to update the solution at each node as follows

\begin{align}
    u^{\star}_j &= u_j^n - \frac{ \Delta t_j\text{Res}_j}{V_j}, \\
    u^{n+1}_j &= \frac{1}{2}(u_j^{\star} + u_j^n) - \frac{1}{2}\frac{ \Delta t_j\text{Res}_j}{V_j},
\end{align}
where $n$ denotes the time level, $u_j^\star$ is the solution produced by the first stage of the RK scheme at node $\bm{p}_j$, and each $\Delta t_j$ is a pseudo-time step which drives the solution towards a converged state. The pseudo-time step at each node is computed by
\begin{align}
    \Delta t_{j} = \frac{(\text{CFL})V_j}{\sum_{\bm{p}_k\in \mathbb{T}_{j}} 
    \frac{1}{2}  | {\bm{n}}_{jk} \cdot \bm{a} |  },
\end{align}
where $\text{CFL}$ is the Courant-Friedrichs-Lewy number. Throughout our numerical experiments, we set $\text{CFL} = 0.5$. 

Accuracy of the edge-based method has been verified in two and three dimensions \cite{barth_AIAA1991,Katz_Sankaran_JCP:2011,diskin_thomas:AIAA2012-0609,nishikawa:AIAA2010,liu_nishikawa_aiaa2016-3969}, but not in four dimensions. This paper presents, for the first time, accuracy verification results in four dimensions. We also provide accuracy verification results in two and three dimensions to highlight the consistency of the method's behavior across multiple dimensions, and to emphasize its versatility. All of these numerical results are summarized in Section~\ref{results_section}.

%In this edge-based finite volume method, we used first-order polynomials to approximate the solution so we expect a second order convergence rate.

%%%%%%%%%%%%%%%%%%%%%%%%%%%%%%%%%%%%%%%%%%%%%%%%

 \subsection{Complexity Analysis}

In this section, we perform a simplified complexity analysis for the edge-based method. Here, we are interested in comparing the complexity of the edge-based method to that of a traditional continuous Galerkin (CG) method. We note that a comparison to traditional discontinuous Galerkin (DG) methods is also possible. Such a comparison may be performed indirectly by comparing edge-based methods to CG methods, and then comparing CG methods to DG methods. Evidently, the comparison between edge-based and CG methods  appears below; in addition, the  comparison between CG and DG methods has already been rigorously explored in the three-dimensional setting (see~\cite{anderson2011petrov,glasby2013comparison} for details).

 The edge-based method is well-known to be an efficient alternative to the $P_1$ CG method (the linear-polynomial CG method) on simplex-element meshes in 2D and 3D~\cite{barth_AIAA1991}. The latter has been successfully used for space-time simulations~\cite{Behr_2008IJNMF,MasudHughes_1997CMAME}. In what follows, we compare the complexities of the two methods and show that the edge-based method is still more efficient in four dimensions. We note that the two methods are mathematically equivalent on simplex-element meshes if the numerical flux is evaluated by using the arithmetic average of the fluxes associated with the two endpoints of an edge~\cite{barth_AIAA1991,Nishikawa_Proof_EB_equiv_P1Galerkin}.
In the remainder of this section, we provide a complexity analysis for this simple case and comment on additional work required for both methods. Throughout our analysis, we will ignore boundary effects  because the computational cost is dominated by calculations associated with interior points of the domain.

 \subsubsection{The $P_1$ CG Residual}

 The $P_1$ CG method applied to the linear advection equation gives the following residual at a node $\bm{p}_j$
\begin{align}
{Res}_j 
=
\frac{1}{d} 
\sum_{T \in \mathbb{T}_{j}^{(d)}}
\left( 
\frac{1}{d+1 } 
\sum_{\bm{p}_i \in T} 
 \flux_i
 \right) \cdot \bm{n}^T_j 
=
\frac{1}{d (d+1) } 
\sum_{T \in \mathbb{T}_{j}^{(d)}}
\left(
\sum_{\bm{p}_i \in T} \flux_i \right) \cdot  \bm{n}^T_j , 
  \label{P1_CG_linear_advection}
\end{align}
 where we recall that $T\in \mathbb{T}_{j}^{(d)}$ denotes a single element which belongs to the set of $d$-simplices sharing the node $\bm{p}_j$, and $\bm{p}_i$ denotes a vertex of $T$. In addition, $\flux_i$ is the flux evaluated at node $\bm{p}_i$, and $\bm{n}_j^T$ is the normal vector of the face opposite to node $\bm{p}_j$. This computation is typically implemented in a loop over elements. In particular, for each vertex $\bm{p}_j$ of the element $T$, we perform the following computation
 \begin{align}
{Res}_j
=
{Res}_j 
+
\left(
\sum_{\bm{p}_i \in T} \flux_i 
 \right) 
\cdot \bm{n}^T_j.
\label{P1_CG_linear_advection_elm_loop}
\end{align}
This element-loop step is followed by a loop over all nodes $\bm{p}_j$, 
%of the element,  <-- Here the loops is over all nodes in the domain.
in which we scale the residual as follows 
 \begin{align}
{Res}_j 
=
\frac{1}{d (d+1) } 
{Res}_j. \label{P1_CG_linear_advection_elm_scaling}
\end{align}
The computational costs of Eqs.\eqref{P1_CG_linear_advection_elm_loop} and \eqref{P1_CG_linear_advection_elm_scaling} can be calculated in a straightforward fashion:

\begin{itemize}
    \item Compute fluxes $\flux_i$ at $d+1$ nodes $\bm{p}_i$ in the element $T$. This step requires $d(d+1)$ multiplications for linear advection problems.
    \item Add together the flux contributions from each of the $d+1$ vertices $\bm{p}_i \in T$. This step requires $d^2$ additions.
    \item Compute the dot product of the sum of fluxes from the previous step with the normal vector of the opposite face $\bm{n}_j^T$. This step requires $d$ multiplications and $(d-1)$ additions.
    \item Perform the dot product from the previous step $(d+1)$ times, one for each node $\bm{p}_j$ of the element $T$. 
    \item Add the result to each node $\bm{p}_j\in T$. This step requires $(d+1)$ additions.
    \item Finally, for every node in the mesh, multiply the residual by a factor of $1/(d(d+1))$.
\end{itemize}
 The total number of additions in this process is given by
\begin{align}
    \left(d^2 + (d-1)(d+1)+ (d+1)\right) N_{elements} = d(2d+1) N_{elements},
\end{align}
and the total number of multiplications is given by
\begin{align}
    \left(d(d+1) + d(d+1)\right) N_{elements} + N_{nodes} = 2d(d+1)N_{elements} + N_{nodes}.
\end{align}
Therefore, the total number of floating point operations (FLOPs) is given by
\begin{align}
   d(2d+1) N_{elements} + 2d(d+1)N_{elements} + N_{nodes}
    =
   d \left( 4 d + 3     \right)  N_{elements}  + N_{nodes}
    . \label{flop_count_one}
\end{align}
 Now, for simplicity, let us  consider a $d$-dimensional CFK triangulation. In this mesh, there are $(d+1)!$ different $d$-simplices which share an interior vertex. As a result, we have that
\begin{align}
    (d+1) N_{elements} = (d+1)! N_{nodes},
\end{align}
and furthermore
\begin{align}
N_{elements} = d! N_{nodes}.
    \label{element_id_d}    
\end{align}
Upon substituting Eq.~\eqref{element_id_d} into Eq.~\eqref{flop_count_one}, we find that the total number of FLOPs is given by
\begin{align}
   d (4d + 3)  d!  N_{nodes}  + N_{nodes} = 
    \left[  d (4d + 3)  d! + 1 \right] N_{nodes}.
    \label{cell_estimate_d}
\end{align}

%Try to express the above in terms of $N_{nodes}$

 \subsubsection{The Edge-Based Residual}

The edge-based method is implemented in a loop over edges. In particular, the following operations are performed on each edge with endpoints $\bm{p}_j$ and $\bm{p}_k$
\begin{align}
{Res}_j = {Res}_j +
\frac{1}{2} ( \flux_j+\flux_k) \cdot \bm{n}_{jk},
\quad
{Res}_k = {Res}_k -
\frac{1}{2} ( \flux_j+\flux_k) \cdot \bm{n}_{jk}.
\label{EB_linear_adv_edge_loop}
\end{align}
We note that there is an implicit scaling of the edge residual by a factor of $1/(d(d+1))$, as this appears in the definition of the directed-hyperarea vector $\bm{n}_{jk}$ (see Eq.~\eqref{vector_identity_conjecture}). Naturally, we observe a direct correspondence between this scaling and the scaling operation in Eq.~\eqref{P1_CG_linear_advection_elm_scaling}. However, the scaling in Eq.~\eqref{P1_CG_linear_advection_elm_scaling} is usually performed every iteration, whereas the scaling associated with the directed-hyperarea vectors is usually a pre-processing step.

The computational cost of Eq.~\eqref{EB_linear_adv_edge_loop} can be calculated as follows:
\begin{itemize}
    \item Compute fluxes $\flux_j$ and $\flux_k$ at the endpoints of each edge. This step requires $2d$ multiplications for linear advection problems.
    \item Add together the flux contributions from each of the 2 vertices which belong to the edge. This step requires $d$ additions.
    \item Compute the dot product of the sum of fluxes from the previous step with the directed-hyperarea vector $\bm{n}_{jk}$. This step requires $d$ multiplications and $(d-1)$ additions.
    \item Add the result to the nodes $\bm{p}_j$ and $\bm{p}_k$ of the edge. This step requires 2 additions.
\end{itemize}

The total number of additions in this process is given by
\begin{align}
    (d + (d-1) + 2)N_{edges} = (2d + 1)N_{edges},
\end{align}
and the total number of multiplications is given by
\begin{align}
    (2d + d) N_{edges} = 3d N_{edges}.
\end{align}
Therefore, the total number of FLOPs is given by
\begin{align}
     (2d + 1)N_{edges} + 3d N_{edges}=
    (5d + 1)N_{edges} 
    . \label{flop_count_two}
\end{align}
 Now, in a $d$-dimensional CFK triangulation there are a maximum of $d!$ different $d$-simplices which share an interior edge. As a result, we have
\begin{align}
    \frac{d(d+1)}{2}N_{elements} = d! N_{edges},
 \label{edge_id_d_00}
\end{align}
or by Eq.~(\ref{element_id_d}), 
\begin{align}
    \frac{d(d+1)}{2}N_{nodes} = N_{edges}.
 \label{edge_id_d}
\end{align}
Finally, upon substituting Eq.~\eqref{edge_id_d} into Eq.~\eqref{flop_count_two}, the  total number of FLOPs becomes
\begin{align}
 \frac{1}{2} d(d+1)( 5d+1) N_{nodes}.
\label{edge_estimate_d}
\end{align}

 \subsubsection{Comparison}

 Upon comparing Eqs.~\eqref{cell_estimate_d} and \eqref{edge_estimate_d}, we see that the edge-based scheme appears to require less FLOPs than the $P_1$ CG scheme by a factor of
\begin{align}
\frac{ 
 \displaystyle 
  \frac{1}{2} d(d+1)(5d+1)  }
{
 \displaystyle 
  d (4d + 3)  d! + 1    }
=
\left\{
\begin{array}{cc}
 \displaystyle  \frac{33}{45} \approx 0.733 & \mbox{for} \quad d=2, \\ [2.5ex]
 \displaystyle   \frac{96}{271}\approx 0.354  &\mbox{for} \quad  d=3, \\  [2.5ex]
 \displaystyle   \frac{210}{1825}\approx 0.115 & \mbox{for} \quad d=4, \\ 
\end{array}
\right.
\end{align}
for a fixed number of nodes $N_{nodes}$.
This verifies that the edge-based method is indeed more efficient (theoretically) than the $P_1$ CG method in 2D and 3D, and shows that it is still more efficient in 4D by a larger factor. 

The key issue is the factorial of the dimension, $d!$, that appears in the complexity of the $P_1$ CG method. This term grows faster than any polynomials of $d$. Another key difference is that the flux is computed only once per edge in the edge-based method, whereas the flux is computed $d+1$ times per element in the $P_1$ CG method. Note also that the vector $\bm{n}^T_j$ has been assumed to be precomputed and stored in the above analysis of $P_1$ CG, but it will require extra memory and pointers from elements to faces. In practice, these vectors are computed on-the-fly in order to reduce memory, and therefore, the actual complexity of $P_1$ CG may be higher. In fairness, the edge-based method is typically implemented with linear solution reconstructions and dissipation terms, requiring gradient calculations. This will increase its complexity. However, the overall complexity may be lower than that of the $P_1$ CG method because of the factorial term $d!$ that grows very rapidly. Also, the $P_1$ CG method will need its own dissipation/stabilization term. A common approach is to add a diffusion term that is discretized by the $P_1$ CG method, which will increase the complexity significantly. Fortunately, this complexity  may be reduced by employing an efficient edge-based formulation recently proposed for the $P_1$ CG discretization of diffusion and viscous terms~\cite{EBV_tetra:AIAAJ2022}.

Finally, it is important to remember that the complexity analysis above comes with some  caveats. In particular, the actual computational cost of a scheme cannot be measured merely by counting FLOPs. Rather, other factors are important, including accuracy per degree-of-freedom, stability on coarse meshes, alignment and frequency of memory accesses, and suitability for optimization on massively-parallel computing architectures. These factors make it nearly impossible to effectively determine the superior efficiency of one scheme relative to another, without taking into account problem-specific and hardware-specific considerations. Nevertheless, the complexity analysis (above) indicates that an edge-based method \emph{may} outperform a classical CG method under certain circumstances.
 
%%%%%%%%%%%%%%%%%%%%%%%%%%%%%%%%%%%%%%%%%%%%%%%%
\section{Numerical Results} \label{results_section}

\subsection{Method of Manufactured Solutions (2D)}

In this section, our goal is to evaluate the order of accuracy of a simple, 2D formulation of the edge-based method. In accordance with standard practice, we solved the following scalar, linear advection equation
\begin{align}
    \frac{\partial u}{\partial x} + \frac{\partial u}{\partial y} = f, \label{space_advection}
\end{align}
where $x$ and $y$ are the two coordinate directions, $u = u(x,y)$ is the scalar solution, and $f= f(x,y)$ is the scalar forcing function. It turns out that we can generate any desired closed-form solution to Eq.~\eqref{space_advection} by choosing the appropriate forcing function. We simply substitute the desired solution $u$ into the left hand side of Eq.~\eqref{space_advection}, evaluate the derivatives, and then substitute the resulting expression in place of the forcing function $f$ on the right hand side. This procedure is often called the `Method of Manufactured' solutions \cite{OberkampfRoy_2010}, as the desired exact solution is generated or `manufactured' by the appropriate choice of a forcing function. 

During the numerical experiments in this study, we selected the following exact-solution functions 
\begin{align}
    u_1 &= x^2 + xy + y^2, \\[1.0ex]
    u_2 &= 3x^2 + 5y^2, \\[1.0ex]
    u_3 &= \exp\left[k(x+y)\right], 
\end{align}
where the constant $k = 0.1$ was chosen for convenience. We can observe that the first function (above) is a symmetric, quadratic polynomial, the second is a non-symmetric quadratic polynomial, and the third is a non-symmetric transcendental function. These functions were chosen due to their distinct and varied features, which enabled us to evaluate the edge-based  method over a wide range of conditions. The corresponding forcing functions for these exact solutions, as derived from Eq.~\eqref{space_advection}, are as follows
\begin{align}
    f_1 &= 3(x+y), \label{forced_one} \\[1.0ex] 
    f_2 &= 6x + 10y, \\[1.0ex]
    f_3 &= 2k \exp\left[k(x+y)\right]. \label{forced_three}
\end{align}
The computational domain for our study was a square: $\Omega = [0,1]^2$. Dirichlet boundary conditions were enforced on each side of the square, and the exact solutions were used to supply the boundary values. In addition, the numerical solutions inside of the square were initialized using the exact solutions. Regular triangular meshes were used to discretize the square. These meshes were generated using the open-source meshing software GMSH~\cite{geuzaine2009gmsh}. Table~\ref{mesh_info} lists the six meshes we tested, and the number of  triangular elements in each mesh. Each successive mesh level quadruples the number of cells relative to the previous level, with 4 cells on the coarsest mesh and over 4,000 cells on the finest. Figure~\ref{fig:meshing} displays four of the six meshes that were used in our study. In addition, Figure~\ref{fig:contours_2d} shows contours of the numerical solutions for the test cases with exact solutions of $u_1$, $u_2$, and $u_3$.
\begin{table}[h!]
\centering
\begin{tabular}{| c | l | l |}
    \hline
    Mesh No. ($n$) & No.~of Cells & No.~of Nodes \\\hline\hline
    0 & 4 & 5 \\\hline
    1 & 16 & 13 \\\hline
    2 & 64 & 41  \\\hline
    3 & 256 & 145  \\\hline
    4 & 1,024 & 545 \\\hline
    5 & 4,096 & 2,113 \\\hline
\end{tabular}
\caption{Family of structured triangular meshes used in the 2D experiments.}
\label{mesh_info}
\end{table}

\begin{figure}[h!]
    \centering
    \includegraphics[width=6cm]{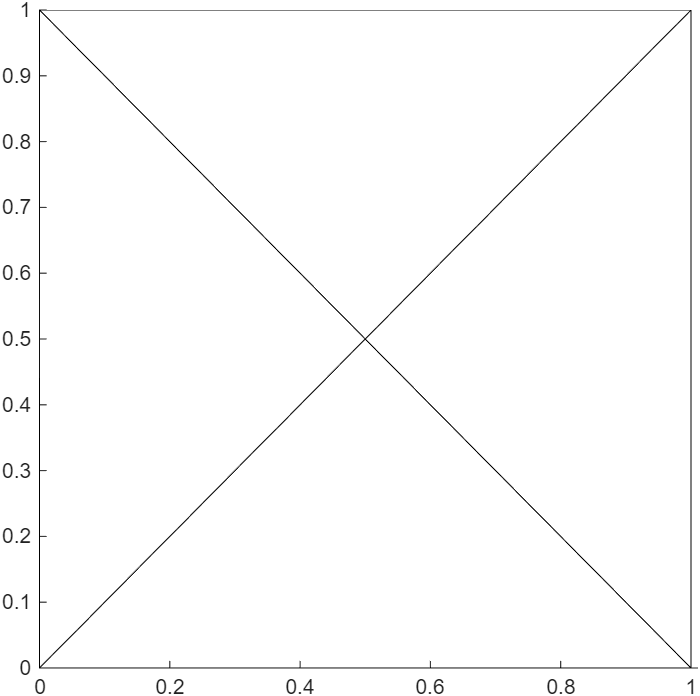} \hspace{1cm}
    \includegraphics[width = 6cm]{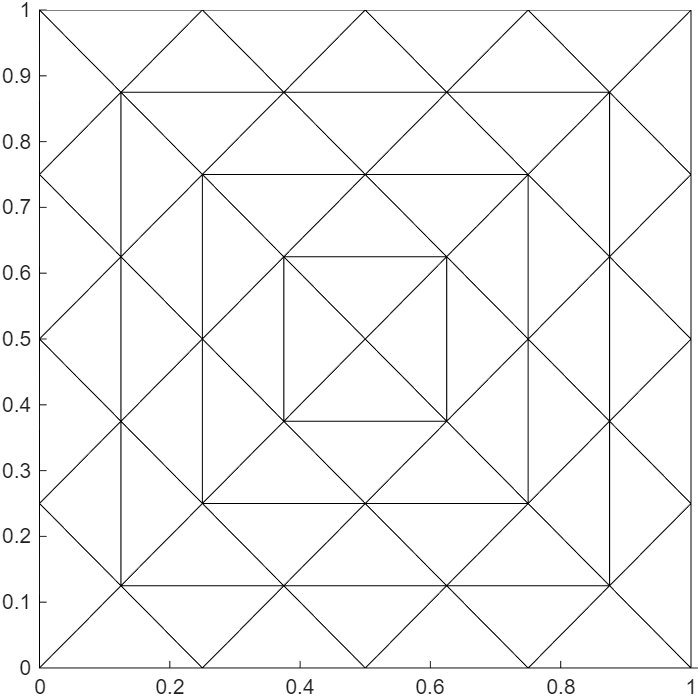} \\[5.0ex]
    \includegraphics[width = 6cm]{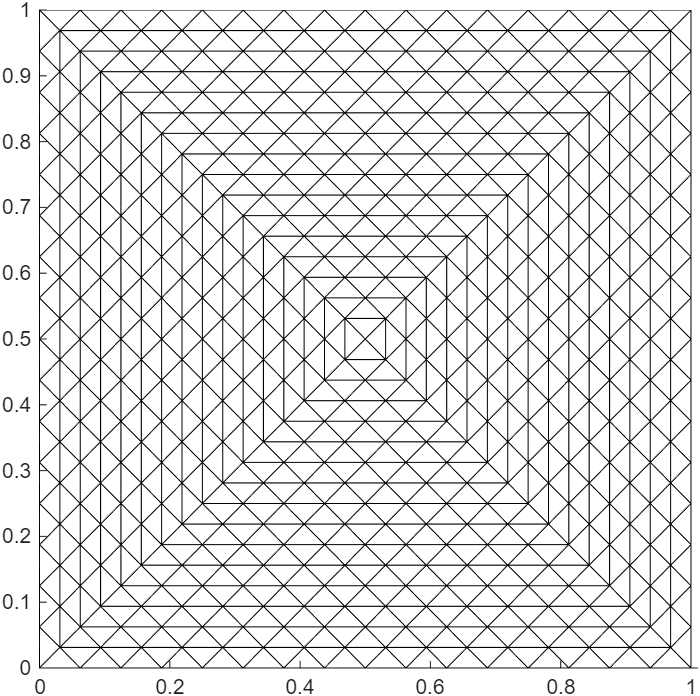} \hspace{1cm}
    \includegraphics[width=6cm]{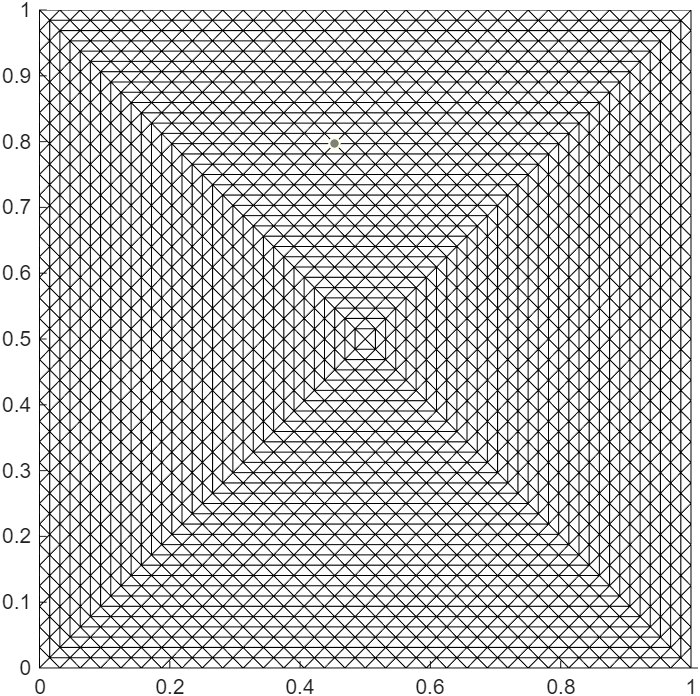}
    \caption{Images of four structured triangular meshes: mesh 0 (top left), mesh 2 (top right), mesh 4 (bottom left), and mesh 5 (bottom right) used in the 2D  experiments.}
    \label{fig:meshing}
\end{figure}

\begin{figure}[h!]
    \centering
    \includegraphics[width=7cm]{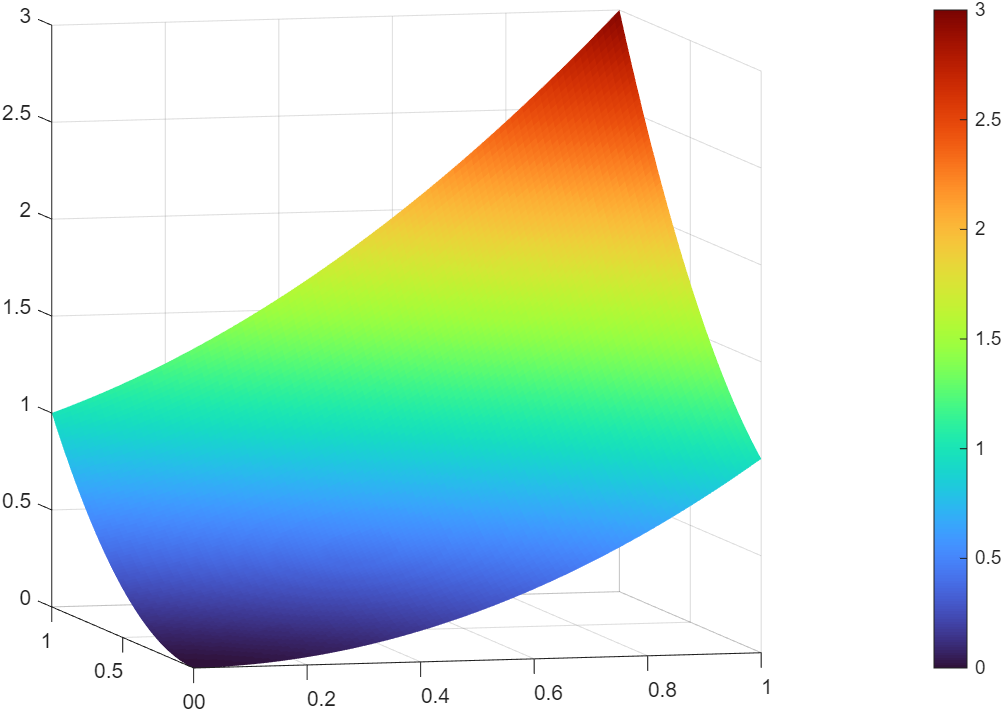} \hspace{1cm}
    \includegraphics[width=7cm]{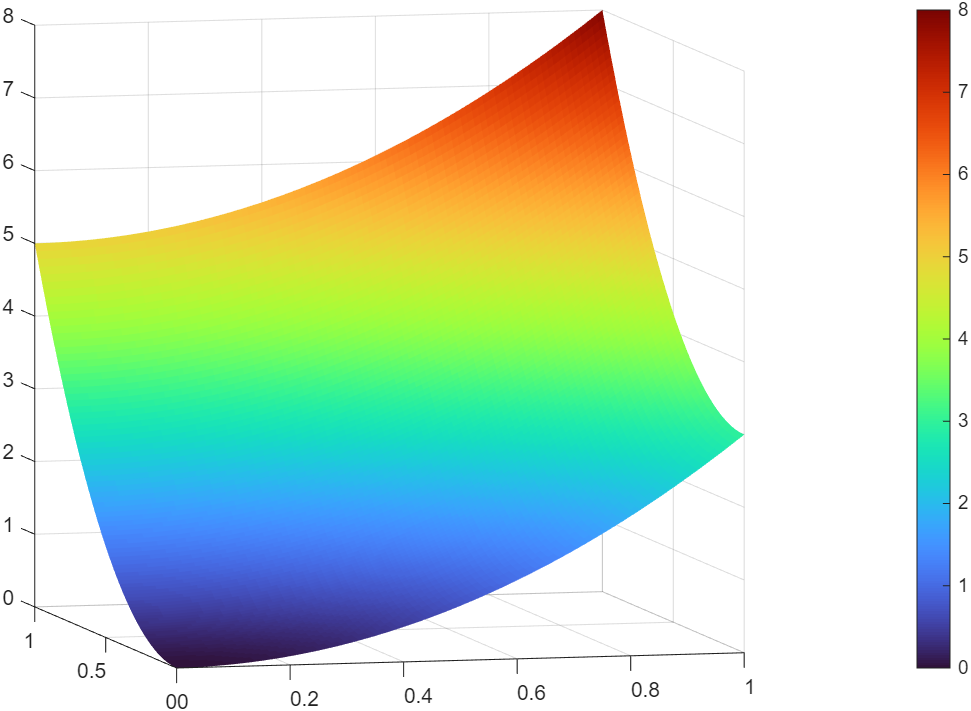} \\[5.0ex]
    \includegraphics[width=7cm]{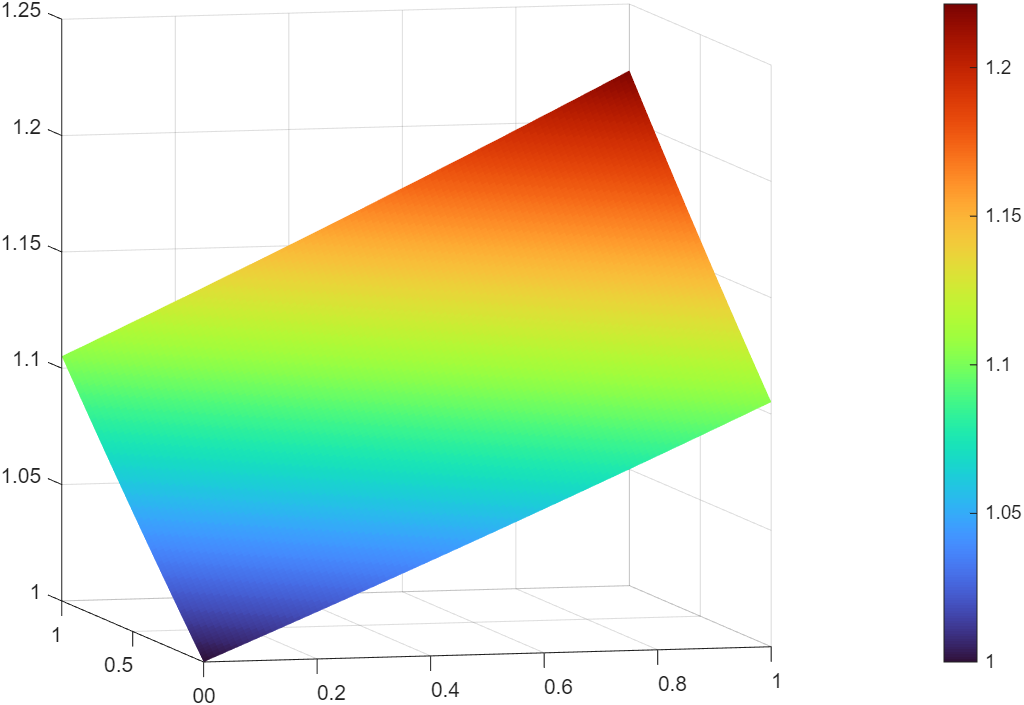}
    \caption{Contour plots of numerical solutions for $u_1$ (top left), $u_2$ (top right), and $u_3$ (bottom) on the 2D mesh with $n = 5$.}
    \label{fig:contours_2d}
\end{figure}
To assess the convergence of our edge-based method, we tested how the error between the exact solution and the numerical solution behaved as we refined the mesh. 
%To quantify error for a given exact solution and forcing function, we looped through all the nodes in each mesh and compared the numerical method's predicted value $u_h$ with the exact solution $u$.
We computed the maximum nodal error by computing the absolute error at each node in the mesh, and thereafter taking the maximum value over all nodes.
%
% \begin{align}
%     \text{error} &= |U_{exact}[i] - U[i] |
% \end{align}
%
Once the maximum nodal error was computed on each mesh, it was tabulated alongside the characteristic element-size  $h$, associated with each mesh. The value of $h$ was computed as $h=\frac{1}{2^{n}}$ where $n$ is the mesh number. The results appear in Table~\ref{tab:errortable}.  For all three cases, the numerical solutions converged to 2nd-order accuracy, with only a few minor deviations. 
\begin{table}[h!]
\centering
%\begin{tabular}{|*{9}{c|}}
\begin{tabular}{|l|c|c|c|c|c|c|}
\hline
     -- & \multicolumn{2}{|c}{Symmetric Quadratic} &  \multicolumn{2}{|c|}{Quadratic} & \multicolumn{2}{|c|}{Exponential} \\ \hline\hline
    Mesh Size ($h$)  & Maximum Error & Slope & Maximum Error & Slope & Maximum Error & Slope \\\hline
    1.0  & 5.49e-1 & -- & 2.00e0 & -- & 3.10e-3 & -- \\\hline
    0.5  & 1.41e-1 & 1.97  & 5.42e-1 & 1.88 & 7.92e-4 & 1.97 \\\hline
    0.25  & 3.78e-2 & 1.90  & 1.47e-1 & 1.88 & 2.14e-4 & 1.89 \\\hline
    0.125  & 9.00e-3 & 2.07 &  3.61e-2 & 2.03 & 5.00e-5 & 2.10 \\\hline
    0.0625  & 2.30e-3 & 1.97 &  9.37e-3 & 1.95 & 1.28e-5 & 1.97 \\\hline
    0.03125  & 5.95e-4 & 1.95 &  2.54e-3 & 1.88 & 3.30e-6 & 1.96 \\\hline
\end{tabular}
\caption{Order of accuracy results for the 2D experiments.}
\label{tab:errortable}
\end{table}

\pagebreak

\subsection{Method of Manufactured Solutions (3D)}

In a similar fashion, we evaluated the order of accuracy for our 3D edge-based method. Following the approach taken in the 2D section, we solved a linear scalar advection equation in accordance with the Method of Manufactured Solutions. We considered the equation
\begin{align}
    \frac{\partial u}{\partial x} +\frac{\partial u}{\partial y} + \frac{\partial u}{\partial z} = f,
\end{align}
where $x$, $y$, and $z$ are the three coordinate directions, $u = u(x,y,z)$ is the scalar solution, and $f = f(x,y,z)$ is the scalar forcing function. For our 3D study, we considered the following exact solutions
\begin{align}
    u_1 &= 1 + x^2 + y^2 + z^2 + xy  + xz + yz, \\
    u_2 &= 1+ x^2 + yz, \\
    u_3 &= 1 + \exp \left[k(x+y+z)\right],
\end{align}
where $k=0.1$. We note that the first function (above) is a symmetric quadratic polynomial, the second function is a non-symmetric quadratic polynomial, and the third function is a non-symmetric transcendental function. \\
The exact-solution functions were manufactured from the following forcing functions
\begin{align}
    f_1 &= 4(x+y+z), \\
    f_2 &= 2x + y + z, \\
    f_3 &= 3k\exp \left[k(x+y+z)\right].
\end{align}
The computational domain was $ \Omega = [0,1]^3 $. The exact solutions were prescribed as Dirichlet boundary conditions on all boundaries of $\Omega$. Furthermore, the solution on the interior of the computational domain was initialized with the exact solution. The domain was tessellated with a uniform mesh of tetrahedral elements using the CFK decomposition. Through this decomposition, the initial domain was divided into $N^3$ cubic elements, and then each cube was subdivided into $6$ tetrahedral elements. Altogether, each mesh contained $6 N^3$ tetrahedral elements. We obtained numerical results for meshes with $N = $ 7, 11, 15, 19, and 23. The details of these meshes are summarized in Table~\ref{mesh_info_3D}.
\begin{table}[h!]
\centering
\begin{tabular}{| c | l | l |}
    \hline
    $N$ & No.~of Cells & No.~of Nodes \\\hline\hline
    7 & 2,058 & 512 \\\hline
    11 & 7,986 & 1,728 \\\hline
    15 & 20,250 & 4,096 \\\hline
    19 & 41,154 & 8,000 \\\hline
    23 & 73,002 & 13,824 \\\hline
\end{tabular}
\caption{Family of structured tetrahedral meshes used in the 3D experiments.}
\label{mesh_info_3D}
\end{table}

Figure~\ref{fig:contours_3d} depicts contours of the numerical solutions for the three test cases with $u_1$, $u_2$, and $u_3$. The maximum nodal errors on each of the five meshes are summarized in Table \ref{tab:errortable_3D}. From these results, it is evident that second-order accuracy is achieved.

\begin{figure}[h!]
    \centering
    \includegraphics[width=7cm]{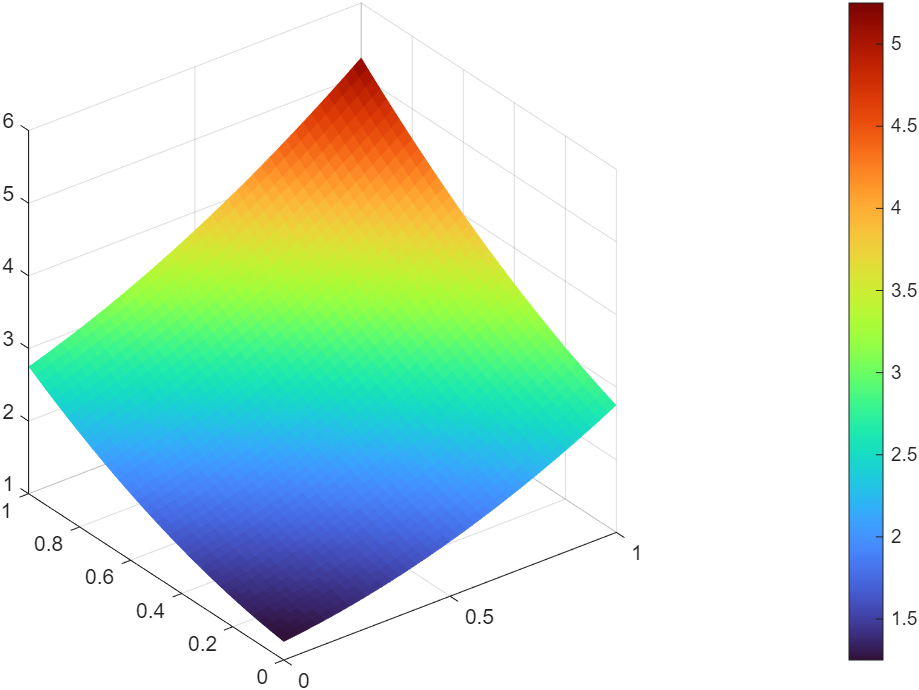} \hspace{1cm}
    \includegraphics[width=7cm]{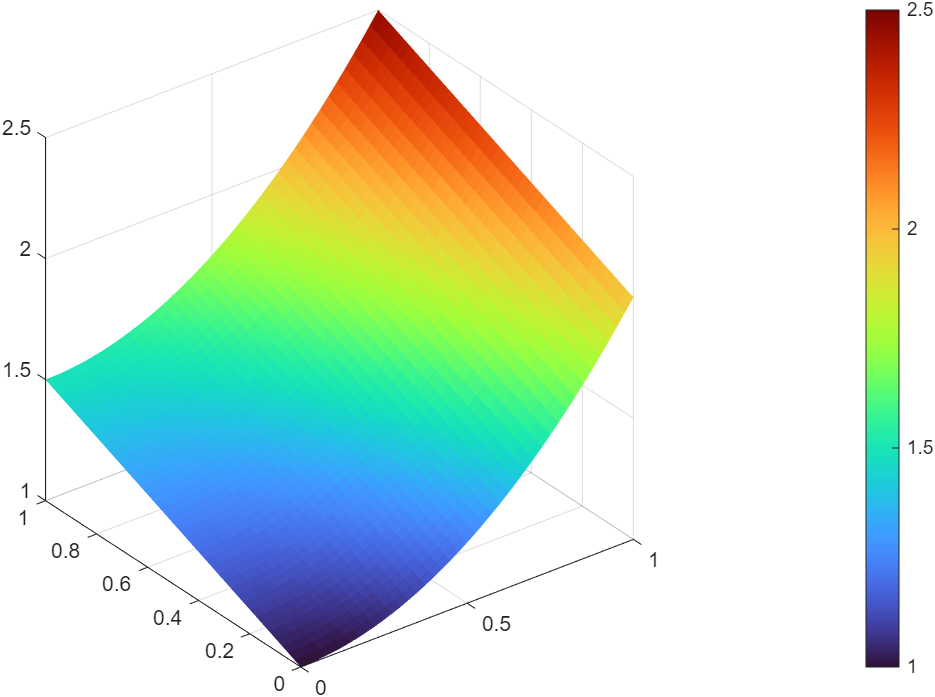} \\[5.0ex]
    \includegraphics[width=7cm]{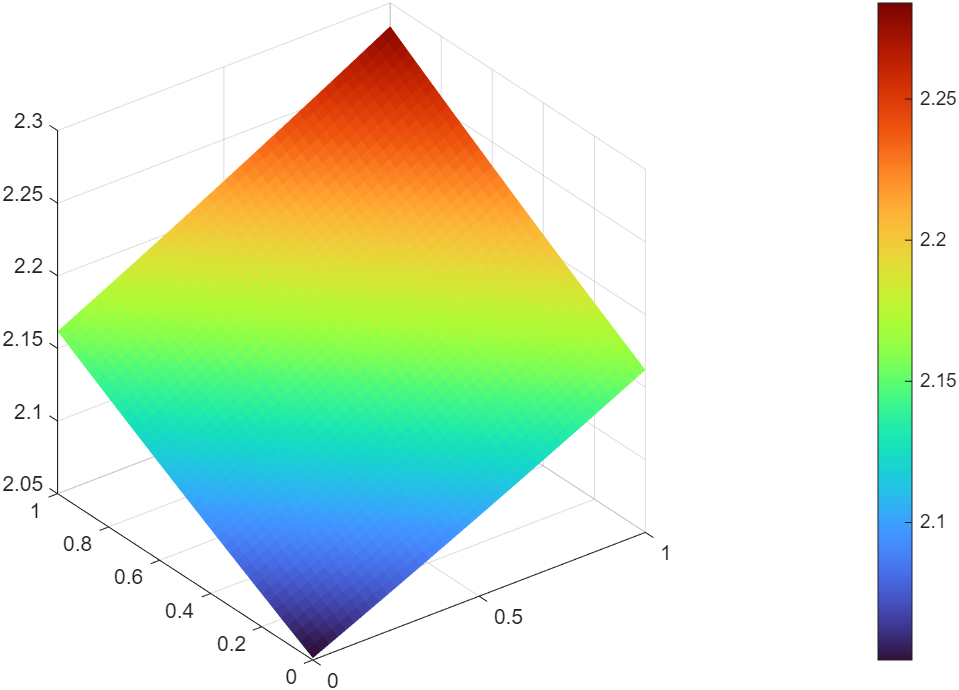}
    \caption{Contour plots of numerical solutions for $u_1$ (top left), $u_2$ (top right), and $u_3$ (bottom) on the 3D mesh with $N = 23$. The plots were created based on a 2D cross-section of the 3D domain taken at $0 \leq x \leq 1$, $0 \leq y \leq 1$, and $z = 0.5$.}
    \label{fig:contours_3d}
\end{figure}

\begin{table}[h!]
\centering
%\begin{tabular}{|*{9}{c|}}
\begin{tabular}{|l|c|c|c|c|c|c|c|c|}
\hline
     -- & \multicolumn{2}{|c}{Symmetric Quadratic} & \multicolumn{2}{|c}{Quadratic} & \multicolumn{2}{|c|}{Exponential} \\ \hline\hline
    Mesh Size ($h$)  & Maximum Error & Slope & Maximum Error & Slope & Maximum Error & Slope \\\hline
    0.5 & 5.49e-2 & -- & 1.57e-2 & -- & 2.38e-4 & -- \\\hline	
0.25	& 2.26e-2 & 1.96 & 7.71e-3 & 1.58 & 1.10e-4 & 1.71 \\\hline
0.125	& 1.21e-2 & 2.02 & 4.34e-3 & 1.85 & 6.17e-5 & 1.87 \\\hline
0.0625	& 7.52e-3 & 2.01 & 2.73e-3 & 1.97 & 3.88e-5 & 1.97 \\\hline
0.04167	& 5.12e-3 & 2.01 & 1.86e-3 & 1.99 & 2.66e-6 & 1.97 \\\hline
\end{tabular}
\caption{Order of accuracy results for the  3D experiments.}
\label{tab:errortable_3D}
\end{table}

\pagebreak

\subsection{Method of Manufactured Solutions (4D)}

Lastly, we performed numerical tests to determine the order of accuracy for our 4D edge-based method. Towards this end, we followed the approach of the previous section, and solved a scalar, linear advection equation 
\begin{align}
    \frac{\partial u}{\partial x} + \frac{\partial u}{\partial y} + \frac{\partial u}{\partial z} + \frac{\partial u}{\partial w}  = f, \label{space_time_advection}
\end{align}
where $x$, $y$, $z$, and $w$ are the four coordinate directions, $u = u(x,y,z,w)$ is the scalar solution, and $f = f(x,y,z,w)$ is the scalar forcing function. We considered the following exact solutions
\begin{align}
    u_1 &= x^2 + y^2 + z^2 + w^2 + xy + xz + xw + yz + yw + zw, \label{exact_one} \\[1.0ex]
    u_2 &= x^2 + 2 y^2 + 3 z^2 + 4 w^2, \\[1.0ex]
    u_3 &= \exp\left[k(x+y+z+w)\right], \label{exact_three}
\end{align}
where $k = 0.025$. We note that the first function (above) is a symmetric quadratic polynomial, the second function is a non-symmetric quadratic polynomial, and the last function is a non-symmetric transcendental function. The associated forcing functions are
\begin{align}
    f_1 &= 5(x+y+z+w), \label{forcing_one} \\[1.0ex]
    f_2 &= 2x + 4y + 6z + 8 w, \\[1.0ex]
    f_3 &= 4 k \exp\left[k(x+y+z+w)\right]. \label{forcing_three}
\end{align}
The computational domain was $\Omega = [0,1]^4$. The exact solutions (Eqs.~\eqref{exact_one}--\eqref{exact_three}) were prescribed as Dirichlet boundary conditions on all boundaries of $\Omega$, and the interior of the domain was initialized with the same functions. The domain was tessellated with a uniform mesh of 4-simplex (pentatope) elements using the CFK decomposition. Here, the initial domain was subdivided into $N^4$ tesseract elements, and each tesseract was subdivided into $24$ conforming pentatope elements. The final meshes contained a total of $24 N^4$ pentatope elements. We obtained our numerical results on meshes with $N = 2$, 4, 8, 16, and~24. The mesh properties are summarized in Table~\ref{mesh_info_4D}. %In a natural fashion, our objective was to obtain numerical approximations of the exact solutions: $u_1$, $u_2$, and $u_3$.
Figure~\ref{fig:contours_4d} shows contours of the numerical solutions for all three test cases. The maximum nodal errors for each case are shown in Table~\ref{tab:errortable_4D}. Here, it is clear that the expected 2nd-order accuracy is obtained. 
\begin{table}[h!]
\centering
\begin{tabular}{| c | l | l |}
    \hline
    $N$ & No.~of Cells & No.~of Nodes \\\hline\hline
    2 & 384 & 81 \\\hline
    4 & 6,144 & 625 \\\hline
    8 & 98,304 & 6,561 \\\hline
    16 & 1,572,864 & 83,521 \\\hline
    24 & 7,962,624 & 390,625 \\\hline
\end{tabular}
\caption{Family of structured pentatopic meshes used in the 4D experiments.}
\label{mesh_info_4D}
\end{table}

\begin{figure}[h!]
    \centering
    \includegraphics[width=7cm]{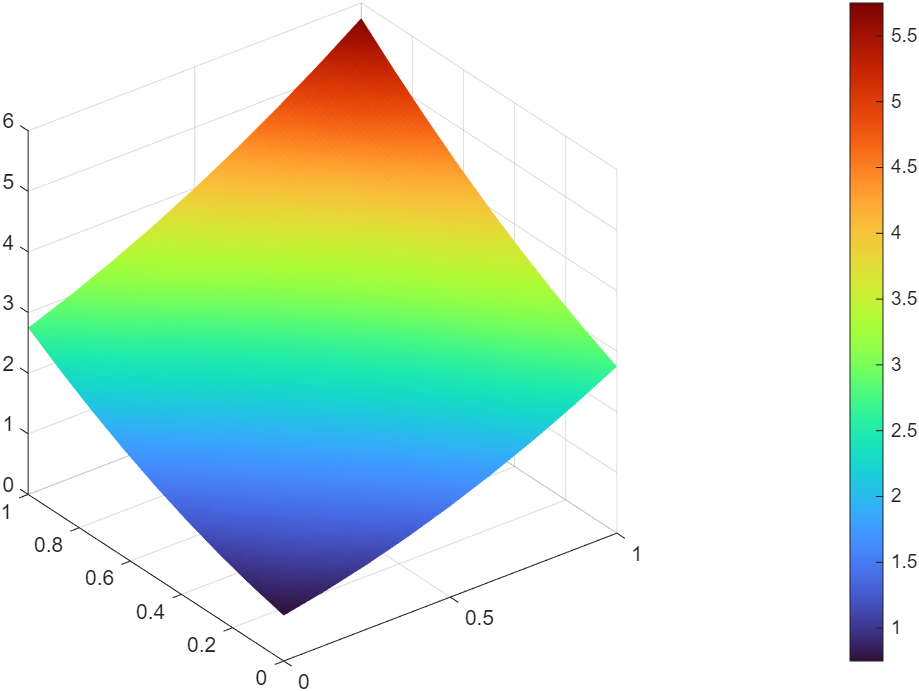} \hspace{1cm}
    \includegraphics[width=7cm]{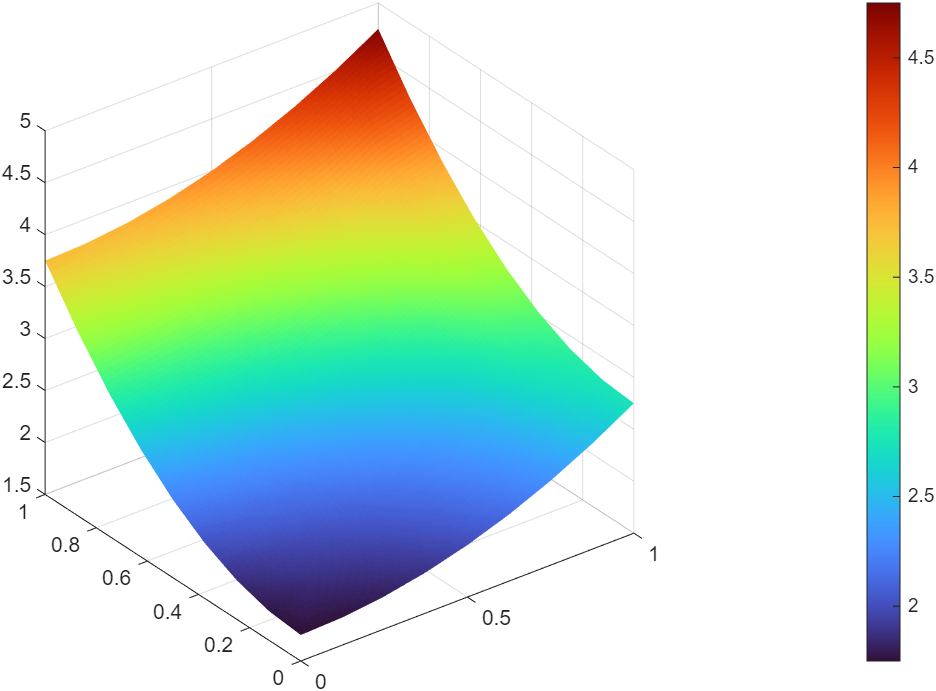} \\[5.0ex]
    \includegraphics[width=7cm]{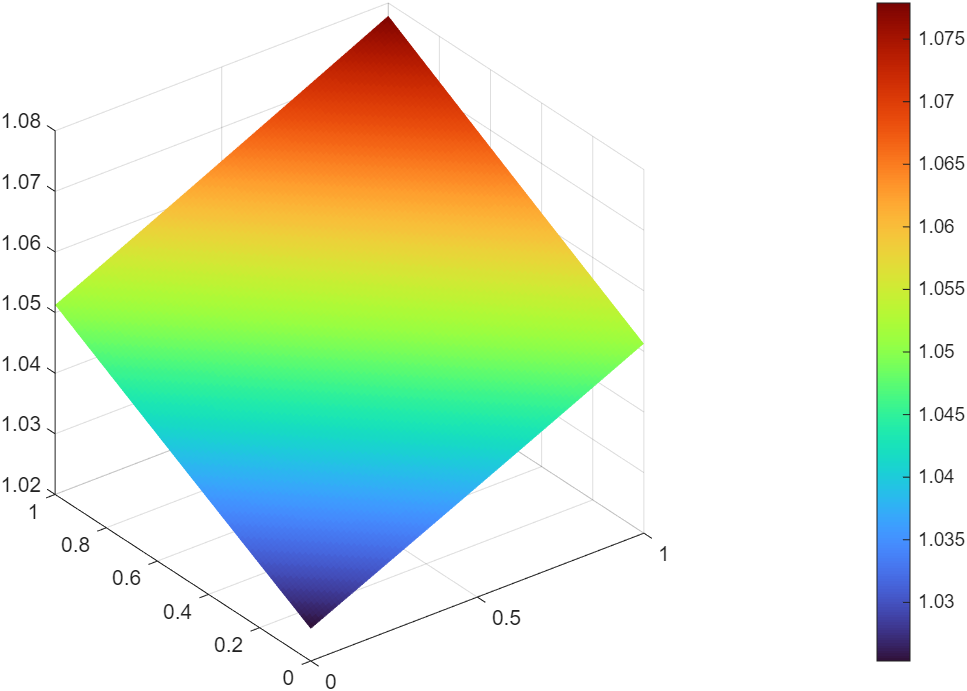}
    \caption{Contour plots of numerical solutions for $u_1$ (top left), $u_2$ (top right), and $u_3$ (bottom) on the structured 4D mesh with $N = 8$. The plots were created based on a 2D cross-section of the 4D domain taken at $0 \leq x \leq 1$, $0 \leq y \leq 1$, $z = 0.5$, and $w = 0.5$.}
    \label{fig:contours_4d}
\end{figure}

\begin{table}[h!]
\centering
%\begin{tabular}{|*{9}{c|}}
\begin{tabular}{|l|c|c|c|c|c|c|c|c|}
\hline
     -- & \multicolumn{2}{|c}{Symmetric Quadratic} & \multicolumn{2}{|c}{Quadratic} & \multicolumn{2}{|c|}{Exponential} \\ \hline\hline
    Mesh Size ($h$)  & Maximum Error & Slope & Maximum Error & Slope & Maximum Error & Slope \\\hline
    0.5 & 8.12e-1 & -- & 1.09e+0 & -- & 3.82e-4 & -- \\\hline	
0.25	& 4.19e-1 & 0.95 & 5.55e-1 & 0.98 & 1.99e-4 & 0.94 \\\hline
0.125	& 1.07e-1 & 1.97 & 1.41e-1 & 1.97 & 5.09e-5 & 1.97 \\\hline
0.0625	& 2.65e-2 & 2.01 & 3.49e-2 & 2.02 & 1.26e-5 & 2.02 \\\hline
0.04167	& 1.17e-2 & 2.02 & 1.54e-2 & 2.02 & 5.56e-6 & 2.02 \\\hline
\end{tabular}
\caption{Order of accuracy results for the  4D experiments on structured meshes.}
\label{tab:errortable_4D}
\end{table}

For the sake of completeness, we repeated the numerical tests with the forcing function $f_3$ on a family of \emph{unstructured}, Delaunay meshes. These meshes were created by taking the point sets from the structured meshes, randomly perturbing the interior points in each direction by  distances less-than-or-equal-to $0.2h$, and then using the perturbed points to create unstructured Delaunay meshes via the Bowyer-Watson algorithm~\cite{bowyer1981computing,watson1981computing}. The properties of the resulting unstructured meshes are summarized in Table~\ref{mesh_info_4D_unstruct}.
\begin{table}[h!]
\centering
\begin{tabular}{| c | l | l |}
    \hline
    $N$ & No.~of Cells & No.~of Nodes \\\hline\hline
    4 & 7,946 & 625 \\\hline
    8 & 133,536  & 6,561 \\\hline
    16 & 2,149,642 & 83,521 \\\hline
    24 & 10,894,264 & 390,625 \\\hline
\end{tabular}
\caption{Family of unstructured pentatopic meshes used in the 4D experiments.}
\label{mesh_info_4D_unstruct}
\end{table}
Figure~\ref{fig:contours_4d_unstruct} shows a contour plot of the numerical solution on the mesh with $N = 8$, and Table~\ref{tab:errortable_4D_unstruct} shows the maximum nodal errors. Again, the expected 2nd-order accuracy is obtained.

\begin{figure}[h!]
    \centering
    \includegraphics[width=7cm]{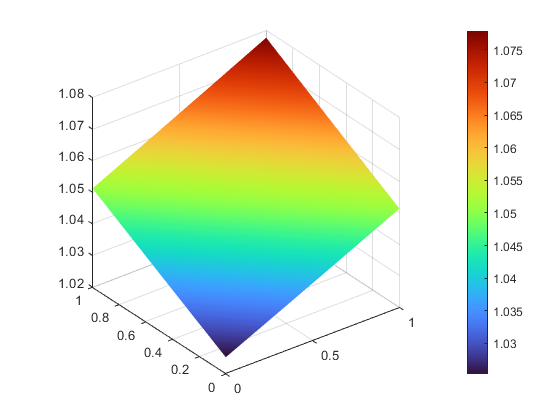} 
    \caption{Contour plots of the numerical solution for $u_3$ on the unstructured 4D mesh with $N = 8$. The plot was created based on a 2D cross-section of the 4D domain taken at $0 \leq x \leq 1$, $0 \leq y \leq 1$, $z = 0.5$, and $w = 0.5$.}
    \label{fig:contours_4d_unstruct}
\end{figure}

\begin{table}[h!]
\centering
\begin{tabular}{|l|c|c|c|c|c|c|c|c|}
\hline
    Mesh Size ($h$)  & Maximum Error & Slope  \\\hline
0.25	& 8.19e-5 & --  \\\hline
0.125	& 2.42e-5 & 1.76 \\\hline
0.0625	& 6.70e-6 & 1.85  \\\hline
0.04167	& 3.01e-6 & 1.97  \\\hline
\end{tabular}
\caption{Order of accuracy results for the 4D experiments on unstructured meshes.}
\label{tab:errortable_4D_unstruct}
\end{table}

\pagebreak

\section{Conclusion} \label{conclusion_section}

In this paper, we provide important geometric formulas for node-centered, edge-based methods on triangulations in $\mathbb{R}^d$. These formulas allow us to compute two key relationships in the median-dual tessellation: i) a relationship between the hypervolume of each median-dual region and the directed-hyperarea vectors, (see Theorem~\ref{hypervolume_theorem}); and ii) a relationship between the directed-hyperarea vectors and the face normals of the associated simplices, (see Conjecture~\ref{directed_hyperarea_conjecture} and  Theorem~\ref{normal_vector_identity_theorem}). The latter relationship holds for all cases in which $d \leq 4$. Our formulas enable us to create a node-centered, edge-based scheme which only stores and operates on edge data in the form of directed-hyperarea vectors---although some limited face data is required for boundary corrections. In particular, we can compute the directed-hyperarea vectors for each edge as a pre-processing step, and then readily calculate the hypervolume of each median-dual region. As a result, computing the residual at each interior node is simply a matter of evaluating the appropriate solution and numerical flux functions (and accumulating any forcing-function contributions), and accessing the corresponding pre-computed directed-hyperarea vectors for each edge.  Our complexity analysis demonstrates that this process has the potential to be computationally inexpensive, relative to classical cell-based methods. Finally, we verified that our edge-based method achieves 2nd-order accuracy in both 2D and 3D, and also, for the first time, demonstrated 2nd-order accuracy in 4D.

Remark~\ref{boundary_remark_03} suggests an explicit method for looping over boundary elements, and enforcing the necessary boundary conditions. However, accuracy-preserving boundary flux quadrature formulas still need to be derived for a node-centered edge-based scheme to achieve second-order and third-order accuracy at boundary nodes in 4D. In addition, there is still interest in extending the proof of Theorem~\ref{normal_vector_identity_theorem} to $d > 4$, and proving Conjecture~\ref{directed_hyperarea_conjecture}. Such endeavors would be of great practical significance for higher-dimensional problems, for example, those in the field of radiation transport. 

\bmsection*{Acknowledgments}

The first and second authors have been supported by the National Institute of Aerospace (NIA) under grant X24-801007-PSU.

\bmsection*{Author Contributions (CRediT)}
\textbf{N. Tufillaro}: Writing – original draft, Writing – review and editing, Visualization, Methodology. \textbf{D. M. Williams}: Writing – original draft, Writing – review and editing, Conceptualization, Visualization, Formal
analysis, Supervision, Funding acquisition. \textbf{N. Nishikawa}: Writing – original draft, Writing – review and editing, Conceptualization, Formal
analysis, Methodology, Supervision, Funding acquisition.

\bmsection*{Data Availability Statement}

The data that support the findings of this study are available from the corresponding author upon reasonable request.

% \bmsection*{Financial disclosure}

% None reported.

% \bmsection*{Conflict of interest}

% The authors declare no potential conflict of interests.

\bibliography{lipics-v2021-sample-article}

% \bmsection*{Supporting information}

% Additional supporting information may be found in the
% online version of the article at the publisher’s website.

%\appendix

\end{document}